\documentclass[12pt]{article}

\usepackage[left=2.7cm,bottom=2.7cm,right=2.7cm,top=2.7cm]{geometry}

\usepackage{amsmath,amsfonts}
\usepackage{array}
\usepackage[caption=false,font=normalsize,labelfont=sf,textfont=sf]{subfig}
\usepackage{textcomp}
\usepackage{stfloats}
\usepackage{url}
\usepackage{verbatim}
\usepackage{graphicx}
\usepackage{cite}
\usepackage{mathtools}
\usepackage{algorithm}
\usepackage{algorithmic}
\usepackage{cases}
\usepackage{enumerate}
\usepackage{tabularx}
\usepackage{multirow}
\usepackage{algorithm}
\usepackage{authblk}
\usepackage{algorithmic}
\usepackage{indentfirst}
\usepackage{setspace}
\usepackage{color}
\usepackage[utf8]{inputenc} 
\usepackage[T1]{fontenc}    
\usepackage{hyperref}       
\usepackage{url}            
\usepackage{booktabs}       
\usepackage{amsfonts}       
\usepackage{nicefrac}       
\usepackage{microtype}      
\usepackage{lipsum}		
\usepackage{graphicx}
\usepackage{natbib}
\usepackage{doi}
\usepackage{amssymb}                  
\usepackage{mathrsfs} 
\usepackage[resetlabels]{multibib}
\usepackage[bottom]{footmisc}
\newtheorem{definition}{Definition}
\newtheorem{lemma}{Lemma}
\newtheorem{theorem}{Theorem}
\newtheorem{corollary}{Corollary}
\newtheorem{proof}{Proof}
\newtheorem{remark}{Remark}
\newtheorem{assumption}{Assumption}
\providecommand{\keywords}[1]{\textit{\quad Key words }:  #1}
\newcommand*{\QEDB}{\null\nobreak\hfill\ensuremath{\square}}
\usepackage[utf8]{inputenc}
\usepackage{tikz}
\usetikzlibrary{patterns}
\usetikzlibrary{decorations.pathreplacing,calligraphy}
\usetikzlibrary {patterns.meta}
\usetikzlibrary{calc}
\tikzdeclarepattern{
  name=mylines,
  parameters={
      \pgfkeysvalueof{/pgf/pattern keys/size},
      \pgfkeysvalueof{/pgf/pattern keys/angle},
      \pgfkeysvalueof{/pgf/pattern keys/line width},
  },
  bounding box={
    (0,-0.5*\pgfkeysvalueof{/pgf/pattern keys/line width}) and
    (\pgfkeysvalueof{/pgf/pattern keys/size},
0.5*\pgfkeysvalueof{/pgf/pattern keys/line width})},
  tile size={(\pgfkeysvalueof{/pgf/pattern keys/size},
\pgfkeysvalueof{/pgf/pattern keys/size})},
  tile transformation={rotate=\pgfkeysvalueof{/pgf/pattern keys/angle}},
  defaults={
    size/.initial=5pt,
    angle/.initial=45,
    line width/.initial=.4pt,
  },
  code={
      \draw [line width=\pgfkeysvalueof{/pgf/pattern keys/line width}]
        (0,0) -- (\pgfkeysvalueof{/pgf/pattern keys/size},0);
  },
}
\begin{document}

\title{Minimax Rates for High-dimensional Double Sparse Structure over $\ell_u(\ell_q)$-balls}

\author[1]{Zhifan Li}
\author[1]{Yanhang Zhang}
\author[1, 2]{Jianxin Yin}

\affil[1]{\footnotesize School of Statistics, Renmin University of China}
\affil[2]{\footnotesize Center for Applied Statistics and School of Statistics, Renmin University of China,
 	\texttt{jyin@ruc.edu.cn} }

\date{}
\maketitle \sloppy

\begin{abstract}
  In this paper, we focus on the high-dimensional double sparse structure, where the parameter of interest simultaneously encourages group-wise sparsity and element-wise sparsity in each group.   
  By combining the Gilbert-Varshamov bound and its variants, we develop a novel lower bound technique for the metric entropy of the parameter space, specifically tailored for the double sparse structure over $\ell_u(\ell_q)$-balls with $u,q \in [0,1]$.
  We prove lower bounds on the estimation error using an information-theoretic approach, leveraging our proposed lower bound technique and Fano's inequality.
  To complement the lower bounds, we establish matching upper bounds through a direct analysis of constrained least-squares estimators and utilize results from empirical processes.
  A significant finding of our study is the discovery of a phase transition phenomenon in the minimax rates for $u,q \in (0, 1]$.
  Furthermore, we extend the theoretical results to the double sparse regression model and determine its minimax rate for estimation error.
  To tackle double sparse linear regression, we develop the DSIHT (Double Sparse Iterative Hard Thresholding) algorithm, demonstrating its optimality in the minimax sense.
  Finally, we demonstrate the superiority of our method through numerical experiments.
\end{abstract}

\keywords{ double sparsity, Gaussian location model, Gilbert-Varshamov bound, iterative hard thresholding, minimax rates.
}

\begingroup\renewcommand\thefootnote{$\dagger$}
\footnotetext{This work is supported by National Key Research and Development Program of China (No. 2020YFC2004900) and the MOE Project of Key Research Institute of Humanities and Social Sciences (22JJD110001).  }

\section{Introduction}

Consider a Gaussian location model (GLM) with $n$ independent observations, described by the equation:
\begin{equation}\label{e1}
Y^i = \theta^* + Z^i,\qquad i \in [n],
\end{equation}
where $Y^i \in \mathbb{R}^{d\times m}$ represents the observation, $\theta^* = (\theta_1, \cdots, \theta_m)\in \mathbb{R}^{d\times m}$ is the parameter of interest, with $\theta_j\in \mathbb{R}^d$ for $j\in [m]$. 
The error term, denoted as $Z \in \mathbb{R}^{d\times m}$, is characterized by independent samples from the normal distribution $\mathcal{N}(0, \sigma^2)$. In this context, the parameter $\sigma$ denotes the level of noise present in the model.
We denote the sample mean by $\bar{Y} = \frac{1}{n}\sum_{i=1}^n Y^i$ and $\bar{Z} = \frac{1}{n}\sum_{i=1}^n Z^i$. Then, model \eqref{e1} is reformulated as 
\begin{equation*}\label{model2}
  \bar{Y} =  \theta^* + \bar{Z},
\end{equation*}
where the entries of $\bar{Z}$ are independently drawn from $\mathcal{N}(0, \frac{\sigma^2}{n})$.
The goal of model \eqref{e1} is to infer the parameter $\theta^*$ given the observations $\{Y^i\}_{i=1}^n$.
GLM is fundamentally important for a variety of statistical problems, such as wavelet Gaussian regression \cite{donoho1994minimax,donoho1995adapting,donoho1998minimax}, model selection \cite{birge2001gaussian,barron1999risk,wu2013model}, and multiple testing problem under sparsity assumption \citep{abramovich2006adapting,abramovich2007optimality}.

Nowadays, various types of sparsity structures are assumed in high-dimensional inference, for example, the element-wise sparsity \cite{T1996, zhang2010, raskutti2011minimax,zhangtong2018, Zhu202014241} and group-wise sparsity \citep{Y2006, zhao2009composite, tsy2011}.
In more general form, many researchers consider assumptions of hard sparsity and soft sparsity\citep{donoho1994minimax,raskutti2011minimax}.
To be specific, hard sparsity means the parameter of interest has an exact number of non-zero elements, while soft sparsity requires the absolute magnitude to decay at a specific rate, and most of the signals may be non-zero.
A widely used approach to capture the notion of these two types of sparsity is $\ell_q$-balls for $q \in [0,1]$.
Suppose that $\eta$ is a $p$-length vector. The $\ell_0$-ball, i.e., in the hard sparsity case, is defined as 
\begin{equation*}
  \mathbb{B}^p_0(s_0) = \{\eta \in \mathbb{R}^p :\ \|\eta\|_{0} = \sum_{i=1}^p \mathbb{I}(\eta_{i} \neq 0) \leq s_0\},
\end{equation*}
where $s_0$ is the exact sparsity level and $\mathbb{I}(\cdot)$ stands for the indicator function. 
On the other hand, for $q \in (0, 1]$, i.e., in the soft sparsity case, the $\ell_q$-ball is defined as
\begin{equation*}
  \mathbb{B}^p_q(R_q) = \{\eta \in \mathbb{R}^p :\ \|\eta\|_q^q = \sum_{i=1}^p | \eta_i|^q \leq R_q\},
\end{equation*}
where $R_q^{1/q}$ is the corresponding radius of the $\ell_q$-ball.
For the case $q=0$, we require the sparsity level $s_0=R_0$ to simplify the notation.
In this paper, we focus our attention on the double sparse structure, that is, simultaneous element-wise and group-wise sparsity. 
We define double sparsity as follows:
\begin{definition}[Double sparsity]\label{df1}
  Given $q \in [0,1]$, $\theta \in \mathbb{R}^{d\times m}$ is called $(s, R_q)$-double sparse when the index set of the columns of $\theta$ belongs to an $\ell_0$-ball while the $j$-th column $\theta_j$ belongs to an $\ell_q$-ball, that is, 
  \begin{equation}\label{e4}
    \|\theta\|_{0,2} = \sum_{j=1}^m \mathbb{I}(\|\theta_{j}\|_2 \neq 0) \leq s\quad  \mbox{and}\quad \theta_j \in \mathbb{B}^d_q(R_q)\ \text{for}\ j \in [m].
  \end{equation}
\end{definition}
The definition of double sparsity can be considered as a combination of $\ell_0$-ball and $\ell_q$-ball, i.e., $\ell_0(\ell_q)$-ball for $q \in [0, 1]$.
Intuitively, the $m$ columns of $\theta$ naturally separate the entries of $\theta$ into $m$ non-overlapping groups, where each group contains $d$ variables.
Double sparsity encourages sparsity across and within the groups simultaneously.
Here we denote the corresponding parameter spaces of \eqref{e4} for $q=0$ and $q \in (0, 1]$ as $\Theta^{m,d}_0(s, s_0)$ and $\Theta^{m,d}_q(s, R_q)$, respectively. 

The parameter spaces of \eqref{e4} can be considered as the Khatri-Rao product of two spaces $\Gamma$ and $\Lambda$.
Let $\Gamma = \{\gamma: \gamma \in \{0, 1\}^m\ \text{and}\ \|\gamma\|_0 \leq s\}$ and $\Lambda \subseteq B^m$ for a given finite set $B \subset \mathbb{R}^d\backslash \{\bf{0}\}$. 
Define the Khatri-Rao product of $\Gamma$ and $\Lambda$ as $\Gamma \odot \Lambda=\{\left(\gamma_1 \otimes \lambda_1, \ldots, \gamma_m \otimes \lambda_m\right): \gamma \in \Gamma \text { and } \lambda \in \Lambda\}$, where $\otimes$ denotes the Kronecker product of two vectors.
The specific forms of  $\Theta^{m,d}_0(s, s_0)$ and $\Theta^{m,d}_q(s, R_q)$ are presented as
\begin{align*}\label{ap12}
  \begin{split}
  &\Theta^{m,d}_0(s, s_0)=\Gamma \odot \Lambda\ \text{for}\ B=\mathbb{B}^d_0(s_0)\backslash \{\bf{0}\},\\
  &\Theta^{m,d}_q(s, R_q)=\Gamma \odot \Lambda\ \text{for}\ B=\mathbb{B}^d_q(R_q)\backslash \{\bf{0}\}.
\end{split}
\end{align*}
Here the index set $\gamma \in \Gamma$ records the locations of the non-zero columns of $\theta$ while 
$\lambda \in \Lambda$ determines the entries of $\theta$.

In addition, a  more challenging scenario is the $\ell_u(\ell_q)$-ball for $u,q \in (0, 1]$, which we define as the double soft sparsity:
\begin{definition}[Double soft sparsity]
    Define the parameter space of $\ell_u(\ell_q)$-ball with radius $R$ for $u,q \in (0, 1]$ as
\begin{equation}\label{lqq}
	\Theta_{u,q}^{m,d}(R) = \left\{\theta \in \mathbb{R}^{m \times d}: \sum_{i=1}^m\left(\sum_{j=1}^d|\theta_{ij}|^{q}\right)^{\frac{u}{q}}\le R\right\}.
\end{equation}
\end{definition}

The parameter space \eqref{lqq} allows us to control the behavior of both the outer sparsity (determined by $u$) and the inner sparsity (determined by $q$) simultaneously, using a common radius $R$. This framework provides a unified approach for investigating the joint sparsity structure in the double soft sparsity setting.
\subsection{Related literature}
In recent years, the double sparse structure has garnered significant attention due to its wide range of applications in various fields. One such application is in genome-wide association studies (GWAS), where genes belonging to the same pathway are grouped together. It is commonly observed that only a small subset of these groups contains causal single nucleotide polymorphisms (SNPs), while the number of non-causal SNPs far exceeds that of causal SNPs \citep{silver2013pathways}. Additionally, the double sparse structure has found applications in classification \cite{rao2015classification,huo2020sgl}, climate data analysis \cite{zhang2020efficient}, and multi-attribute graph estimation \cite{tugnait2021sparse}, among others.

Due to the prevalence of the double sparse structure in various applications, numerous computationally-feasible methods have been developed to tackle this problem.
A common approach to address this issue is through the use of a composition of two penalties, known as bi-level selection \cite{breheny2009penalized}.
Within the framework of bi-level selection, several composite methods have been proposed. For instance, \cite{huang2009group} introduced the group bridge method, which applies a bridge penalty at the group level and an $\ell_1$ penalty at the individual level.
They demonstrated the group selection consistency of the group bridge under certain regularity conditions.
Additionally, the composite MCP \cite{huang2012selective} and group exponential lasso \cite{breheny2015group} were proposed for bi-level selection.
\cite{simon2013sparse} extended the ordinary Lasso \cite{T1996} and group lasso \cite{Y2006} to the double sparse case by simultaneously incorporating these two penalties, naming it as sparse group lasso.
To expedite the convergence of sparse group lasso, notable approaches have been developed \citep{ida2019fast, zhang2020efficient}.
However, despite their promising performance in practical applications \cite{breheny2009penalized, chatterjee2012sparse, zhang2020efficient}, these methods lack rigorous theoretical guarantees. Therefore, comprehensively analyzing and understanding the theoretical properties of these methods remain open challenges that need to be addressed.

Recently, \cite{cai2019sparse} conducted an in-depth statistical analysis of the sparse group lasso, establishing minimax lower bounds on the estimation error for double sparse linear regression. They also demonstrated that the sparse group lasso achieves an optimal minimax rate of convergence under moderate settings. However, their focus was primarily on the double sparse structure over the $\ell_0(\ell_0)$-ball, rather than considering more general sparsity structures such as soft sparsity. Notably, \cite{raskutti2011minimax} extensively investigated the minimax rates of estimation error for high-dimensional linear regression models over $\ell_q$-balls, providing theoretical insights into the general sparsity case.

Efficient methods for estimating high-dimensional sparse structures remain a highly active area of research in statistics and machine learning. Among the most widely studied methods are the lasso \cite{T1996} and its variants \citep{Y2006, jacob2009group, friedman2010note, simon2013sparse}. However, these convex estimators suffer from inherent estimation bias. To address this limitation, non-convex methods such as Iterative Hard Thresholding (IHT) have gained attention. IHT and its variants have shown promise in solving high-dimensional linear regression problems (where the dimensionality $p$ is much larger than the sample size $n$) \citep{blumensath2009iterative, jain2014iterative, ndaoud2020scaled}. In IHT, given the sparsity level $s$, a gradient descent step is initially applied to the parameter, followed by retaining the $s$ largest absolute values at each step.

Recently, several studies on Iterative Hard Thresholding (IHT) have shown that the IHT technique achieves the optimal rate for estimation error \citep{jain2014iterative, liu2020between, ndaoud2020scaled}.
In particular, \cite{ndaoud2020scaled} introduced a novel IHT procedure that dynamically updates the threshold geometrically at each step until it reaches a universal statistical threshold, rather than maintaining a fixed number of non-zero components as in standard IHT.
Building on this idea, they consider a specific sequence of thresholds that leads to the minimax optimality of the estimator.
Motivated by these advancements, we extend this IHT-style method to address high-dimensional double sparse linear regression and investigate the corresponding statistical guarantees.
\subsection{Our contributions}
The main contributions of our work can be summarized as follows:
\begin{itemize}
  \item  We develop a novel approach to construct the packing set of the parameter space using a combination of binary and Q-ary Gilbert-Varshamov bounds \citep{gilbert1952comparison}. This approach is applicable to estimation problems for simultaneously structured models, where the parameter of interest exhibits multiple sparse structures simultaneously. This construction is detailed in Lemma \ref{lem2}.
  \item Leveraging this novel technique, we establish minimax lower bounds over $\ell_u(\ell_q)$-balls for $u, q \in [0, 1]$. Additionally, we derive matching upper bounds using empirical process techniques. Importantly, we identify phase transition boundaries in the optimal error rates for $u, q \in (0, 1]$. These phase transitions depend on the relationships between $u$ and $q$, as well as between $m$ and $d$. The optimal rates are divided into three distinct regimes, each exhibiting different error behaviors. To the best of our knowledge, our work is the first to discover the phase transition phenomenon of double soft sparsity.
  \item We propose a novel algorithm called double sparse iterative hard thresholding (DSIHT) for solving high-dimensional double sparse linear regression problems. We prove that our proposed method achieves optimality in the minimax sense. Furthermore, our simulation studies demonstrate the superiority of our approach compared to traditional IHT and its variants for group selection under double sparsity.
\end{itemize}

\subsection{Organization of the paper}
The remainder of this paper is organized as follows.
After a brief discussion on notations and preliminaries in Section \ref{subsec21}, we present the lower bounds on $\ell_2$-loss over $\ell_0(\ell_q)$-balls in Section \ref{subsec1}, followed by the matching upper bounds in Section \ref{subsec2}.
In Section \ref{lqlq}, we extend the theoretical results for minimax rates to the case of $l_u(l_q)$-balls.
We then study the application of the double sparse structure in double sparse linear regression in Section \ref{sec3}.
First, we prove the optimal minimax rate of estimation error in Section \ref{subsec31}.
Next, we develop the DSIHT algorithm and prove its optimality in the minimax sense in Section \ref{subsec32}.
Section \ref{numerical} presents numerical experiments demonstrating the superiority of our proposed DSIHT algorithm over the traditional IHT and group IHT algorithms.
In Section \ref{sec4}, we provide the proof sketches of our main theorems, while all technical lemmas and detailed proofs can be found in the Appendix.
Finally, we summarize our study in Section \ref{conclusion}. 

\section{Main Results}\label{sec2}
\subsection{Notations and preliminaries}\label{subsec21}
In the rest of this paper, we use the following notations. For the given sequences $a_n$ and $b_n$, we say that $a_n = O(b_n)$ (resp $a_n  = \Omega(b_n)$) when $a_n \le cb_n$ (resp $a_n \ge c b_n$) for some
positive constant $c$. We write $a_n \asymp b_n$ if $a_n = O(b_n)$ and $a_n  = \Omega(b_n)$.
Denote $[m]$ as the set $\{1,2,\ldots,m\}$.
Denote $\mathbf{I}_p$ as the $p\times p$ identity matrix.
Denote $x \wedge y $ as the minimum of $x$ and $y$.
For a set $A$, denote $|A|$ as the cardinality of $A$.
For a vector $\eta$, denote $\|\eta\|_2$ as its Euclidean norm.
For a matrix $\theta$, denote $\|\theta\|_F$ as its Frobenius norm.

To evaluate the performance of an estimator $\hat \theta$, (i.e., a measurable function of observations $\{Y^i\}_{i=1}^n$).
It is common to define a loss function $\mathcal{L}(\cdot, \cdot): \mathbb{R}^{d\times m} \times \mathbb{R}^{d\times m} \rightarrow \mathbb{R}$ and analyze the loss $\mathcal{L}(\hat\theta, \theta^*)$.
We define the loss function $\mathcal{L}(\cdot, \cdot)$ as 
\begin{equation*}
  \mathcal{L}(\hat\theta, \theta^*) = \|\hat \theta - \theta^*\|^2_F\quad \text{for all}\ \hat\theta,\theta^* \in \mathbb{R}^{d \times m}.
\end{equation*} 
In this paper, we call the estimation error measured by the $F$-norm as the $\ell_2$-loss for the matrix.
This section is devoted to establishing the lower and upper bounds on the following minimax risk:
\begin{align*}
  &\mathcal{M}(\Theta^{m,d}_0(s, s_0))=\min_{\hat \theta} \max_{\theta^* \in \Theta^{m,d}_0(s, s_0)}E\left[ \|\hat\theta-\theta^*\|_F^2\right],\\
  &\mathcal{M}(\Theta^{m,d}_q(s, R_q))=\min_{\hat \theta} \max_{\theta^* \in \Theta^{m,d}_q(s, R_q)}E\left[ \|\hat\theta-\theta^*\|_F^2\right].
\end{align*}

The concept of covering and packing numbers play an important role in our remaining analysis. 

\begin{definition}[Covering and Packing Numbers, \cite{raskutti2011minimax}] Consider a compact metric space consisting of a set $\mathcal{S}$ and a metric $\rho$ : $\mathcal{S} \times \mathcal{S} \rightarrow \mathbb{R}^{+}$
\begin{itemize}
  \item  An $\epsilon$-covering of $\mathcal{S}$ with respect to the metric $\rho$ is a collection $\left\{\theta^{1}, \ldots, \theta^{N}\right\} \subset \mathcal{S}$ such that for all $\theta \in \mathcal{S}$, there exists some $i \in\{1, \ldots, N\}$ with $\rho\left(\theta, \theta^{i}\right) \leq \epsilon$. The $\epsilon$-covering
number $N(\epsilon ; \mathcal{S}, \rho)$ is the cardinality of the smallest $\epsilon$-covering.
  \item A $\delta$-packing of $\mathcal{S}$ with repsect to the metric $\rho$ is a collection $\left\{\theta^{1}, \ldots, \theta^{M}\right\} \subset \mathcal{S}$ such that $\rho\left(\theta^{i}, \theta^{j}\right)>\delta$ for all distinct $i, j$. The $\delta$-packing number $M(\delta ; \mathcal{S}, \rho)$ is the cardinality of the largest $\delta$-packing.
\end{itemize}
\end{definition}

Covering and packing numbers provide essentially the same measure of the massiveness of a set.
In particular, the relation between covering number and packing number is described as $M(2\epsilon ; \mathcal{S}, \rho) \leq N(\epsilon ; \mathcal{S}, \rho) \leq M(\epsilon  ; \mathcal{S}, \rho)$.
These two quantities exhibit the same scaling behavior as $\epsilon \rightarrow 0$.
Additionally, the logarithm of the covering number $\log N(\epsilon ; \mathcal{S}, \rho)$ is known as the metric entropy of $\mathcal{S}$ with respect to $\rho$.

\begin{definition}[entropy number of $\ell_q$ ball \cite{triebel2010fractals,kuhn2001lower}]
  Consider a quasi-Banach space consisting a compact set $\mathcal{S}$ and a quasi-metric $\rho$. $N(\epsilon;\mathcal{S},\rho)$ denotes the covering number with radius $\epsilon$. 
 For $k=1,2,\ldots$ the dyadic entropy number is defined as
 $$
 \epsilon_{k}(\mathcal{S},\rho) \coloneqq \inf\{\epsilon>0: N(\epsilon;\mathcal{S},\rho) \le 2^{k-1}\}.
 $$
 Consider a $p$-dimensional vector space. Suppose $\mathcal{S}$ is a $\ell_q$ unit-ball and $\rho$ is the metric induced by $\ell_2$-norm. Then, we have the following  Sch{\"u}tt’s  Theorem:
 \begin{subnumcases} 
   {\epsilon_k(\mathbb{B}_{q}^{p}(1),\|\cdot\|_2)\asymp}
   1 & for $1\le k \le \log p$ \label{wzero}\\
   \left(\frac{\log(\frac{p}{k}+1)}{k}\right)^{\frac{1}{q}-\frac{1}{2}} & for $\log p \le k \le p $ \label{high} \\
   2^{-\frac{k-1}{p}}p^{\frac{1}{2}-\frac{1}{q}} & for $k\ge p$. \label{low}
 \end{subnumcases}
\end{definition}

By analogy with  \cite{raskutti2011minimax}, for the case $q = 0$, we require that $m \ge 4s \ge c_1$ and $d \ge 4s_0 \ge  c_2$ as well. For $q \in (0, 1]$, we add some reasonable assumptions that for some constants $C_1, C_2>0$ and $\delta  \in (0, 1)$, which requires that the triple $(n, d, R_q)$ satisfies
\begin{equation}\label{eq:Rq}
 \frac{d}{R_{q} n^{q / 2}} \geq C_{1} d^{\delta} {\geq} C_{2}.
\end{equation}
\begin{remark}
 We clarify the assumption \eqref{eq:Rq} in two aspects:
 \begin{itemize}
   \item [(1)] Our interest is high-dimensional regime where both $m$ and $d$ are much larger than $n$.
   \item [(2)] 
   Assumption \eqref{eq:Rq} matches the rate of the entropy number in \eqref{high}, which makes sense in high-dimensional sparsity. 
   Therefore, we can avoid the trivial situations i.e., \eqref{wzero} and \eqref{low}.
 \end{itemize}
\end{remark}

\subsection{Lower Bounds on $\ell_2$-loss}\label{subsec1}
Our first main result establishes the minimax lower bounds of estimation error over $\ell_0(\ell_q)$-balls.

\begin{theorem}\label{th1}
  Consider model \eqref{e1} under the double sparse structure.
  \begin{itemize}
    \item [(a)] Case $q=0$: 
    For any measurable estimator $\hat{\theta}$, the lower bound of estimation error
    \begin{equation}\label{e5}
     \inf_{\hat{\theta}} \sup_{\theta\in \Theta^{m,d}_0(s, s_0)}P\left(\|\hat{\theta}-\theta\|_F^2\geq C_{\ell}\frac{\sigma^2}{n}( s \log \frac{em}{s}+ss_0\log \frac{ed}{s_0})\right) \geq \frac{1}{2},
    \end{equation}
where $C_{\ell}$ is some positive constant. The expectation form of the minimax lower bound is
$$
\mathcal{M}(\Theta^{m,d}_0(s, s_0)) \geq \frac{1}{2}C_{\ell}\frac{\sigma^2}{n}( s \log \frac{em}{s}+ss_0\log \frac{ed}{s_0}).
$$
\item [(b)] Case $q \in (0, 1]$: 
For any measurable estimator $\hat{\theta}$, the lower bound of estimation error
    \begin{equation}\label{e6}
 \inf_{\hat{\theta}}  \sup_{\theta\in \Theta^{m,d}_q(s, R_q)}P\left(\|\hat{\theta}-\theta\|_F^2\geq C_{\ell}\{\frac{\sigma^2}{n}s\log\frac{em}{s}+s R_q(\frac{\sigma^2 }{n}\log d)^{1-\frac{q}{2}}\}\right) \geq \frac{3}{8},
    \end{equation}
where $C_{\ell}$ is some positive constant. The expectation form of the minimax lower bound is
$$
\mathcal{M}(\Theta^{m,d}_q(s, R_q)) \geq \frac{3}{8}C_{\ell}\{\frac{\sigma^2}{n}s\log\frac{em}{s}+s R_q(\frac{\sigma^2 }{n}\log d)^{1-\frac{q}{2}}\}.
$$
  \end{itemize}
\end{theorem}

Theorem \ref{th1} shows that the estimation error of $\theta^*$ involves a sum of two quantities: a term involving the complexity of identifying the non-zero columns (groups) and a term that corresponds to the complexity
of estimating parameters over each group.
Intuitively speaking, the first term, $s\log \frac{em}{s}$, in the lower bounds corresponds to the complexity of capturing $s$ non-zero columns while
the second term in $\eqref{e5}$ or $\eqref{e6}$ is related to the complexity of estimating parameter over $\ell_q$-balls ($q$=0 or $q$>0).

Moreover, when $m = s = 1$, the lower bounds in Theorem \ref{th1} reduce to $\Omega(s_0\log \frac{ed}{s_0})$ and $\Omega(R_q(\frac{\sigma^2}{n}\log d)^{1-\frac{q}{2}})$,
which are consistent with the lower bounds for recovery of sparse vectors over $\ell_q$-balls \cite{donoho1994minimax, raskutti2011minimax}.
By setting $d=s_0$, Theorem \ref{th1}(a) recovers the lower bound for estimating the non-overlapping group structure \cite{tsy2011}.

\begin{remark}
	Part (a) of Theorem \ref{th1} derives the same minimax lower bounds of estimation error as \cite{cai2019sparse}.
	Notably, the parameter space $\Theta^{m,d}_0(s, s_0)$ is slightly different from \cite{cai2019sparse}, where \cite{cai2019sparse} constrains the number of the overall nonzero elements rather than the nonzero elements in each group.
	In this work, we show that it is convenient to extend the lower bound technique from $\ell_0$-ball to $\ell_q$-ball over the parameter space $\Theta^{m,d}_q(s, R_q)$.
\end{remark}

\begin{remark}
  For $q \in (0, 1]$, we assume $R_q = s_0 \delta^q$. By restricting $s_0$ to the interval $[1, d^v]$ where $v \in (0, 1]$, we obtain the following inequality:
  \begin{equation*}
  C_1 n R_q^{\frac{2}{q}} d^{v\left(1-\frac{2}{q}\right)} \stackrel{(i)}{\leq} \log \frac{em}{s} \stackrel{(ii)}{\leq} C_2 n R_q^{\frac{2}{q}}.
  \end{equation*}
  If condition $(i)$ is violated, based on \eqref{eq:Rq}, we can observe that the first term is dominated by the second term in \eqref{e6}.
  On the other hand, if condition $(ii)$ is violated, the lower bound for $q \in (0, 1]$ is degenerate, resulting in:
  \begin{equation*}
    \mathcal{M}(\Theta^{m,d}_q(s, R_q)) \geq s R_q^{\frac{2}{q}},
  \end{equation*}
  which implies that the zero estimator is optimal in the minimax sense.
\end{remark}

\subsection{Upper Bounds on $\ell_2$-loss}\label{subsec2}

In this section, we turn to the analysis of the corresponding upper bounds on the $\ell_2$-norm minimax rate over $\ell_0(\ell_q)$-balls.
Here we consider the constrained least-squares estimators over the parameter spaces $\Theta^{m,d}_q(s, R_q)$:
\begin{align}\label{e7}
  \begin{split}
  \hat \theta_q \in \arg \min_{\theta \in \Theta^{m,d}_q(s, R_q)} \|\bar Y-\theta\|_F^2.
\end{split}
\end{align}

\begin{theorem}\label{th2}
  Consider model \eqref{e1} under the double sparse structure.
  \begin{itemize}
    \item [(a)] Case $q=0$: 
    Given estimator $\hat{\theta}_0$ defined in \eqref{e7} and any  $\epsilon^2 \geq C_u\frac{\sigma^2}{n}( s \log \frac{em}{s}+ss_0\log \frac{ed}{s_0})$, we have the estimator error 
\begin{equation}\label{eq:u1}
      \sup_{\theta \in \Theta^{m,d}_0(s, s_0)}\|\hat{\theta}_0-\theta\|_F^2\leq \epsilon^2,
\end{equation}
 holds with probability greater than $1-C_1\exp \{-C_2n\epsilon^2\}$ for some positive constants $C_u, C_1$ and $C_2$.
    \item [(b)] Case $q \in (0, 1]$: 
Given estimator $\hat{\theta}_q$ defined in \eqref{e7} and any  $\epsilon^2 \geq  C_u\{\frac{\sigma^2}{n}s\log\frac{em}{s}+s R_q(\frac{\sigma^2 }{n}\log d)^{1-\frac{q}{2}}\}$, we have the estimator error 
\begin{equation}\label{eq:u2}
      \sup_{\theta \in \Theta^{m,d}_q(s, R_q)}\|\hat{\theta}_q-\theta\|_F^2\leq \epsilon^2,
\end{equation}
 holds with probability greater than $1-C_1\exp \{-C_2n\epsilon^2\}$ for some positive constants $C_u, C_1$ and $C_2$.
  \end{itemize}
\end{theorem}

Theorem \ref{th2} establishes the matching upper bounds of the estimation error, and we can also derive the expectation form which implies that the lower bounds in Theorem \ref{th1} are tight. Additionally, Theorem \ref{th2} yields that the constrained estimators defined in \eqref{e7} are
rate-optimal. The results of Theorem \ref{th1} and \ref{th2} together show the minimax optimal rate up to constant factors.
	\begin{remark}
	The results of Theorem \ref{th1} and \ref{th2} together show the minimax optimal rate up to constant factors.
	In specific, the minimax rate for $q=0$ scales as 
	\begin{equation*}
		\mathcal{M}(\Theta^{m,d}_0(s, s_0)) \asymp \left\{\frac{\sigma^2}{n}( s \log \frac{em}{s}+ss_0\log \frac{ed}{s_0})\right\},
	\end{equation*}
	and for $q \in (0, 1]$, the minimax rate scales as
	\begin{equation*}
		\mathcal{M}(\Theta^{m,d}_q(s, R_q)) \asymp \left\{\frac{\sigma^2}{n}s\log\frac{em}{s}+s R_q(\frac{\sigma^2 }{n}\log d)^{1-\frac{q}{2}}\right\}.
	\end{equation*}
\end{remark}

\subsection{Minimax rates over $\ell_{u}(\ell_{q})$-balls}\label{lqlq}
To avoid over-complicated scenarios, we assume 
\begin{equation}\label{lqqcondition}
	\left\{\begin{array}{ll}
		\log m< R(\frac{n}{\log d})^{\frac{u}{2}} \le \frac{m}{\log m}\\
		\log m< R(\frac{n}{\log m})^{\frac{u}{2}} \le \frac{m}{\log m}\\
		\log d< R^{\frac{q}{u}}(\frac{n}{\log d})^{\frac{q}{2}} < \frac{d}{\log d}.
	\end{array}
	\right.
\end{equation}
In general, three important quantities, namely $R(\frac{n}{\log d})^{\frac{u}{2}}, R(\frac{n}{\log m})^{\frac{u}{2}}, R^{\frac{q}{u}}(\frac{n}{\log d})^{\frac{q}{2}}$, asymptotically equal to $m^{\epsilon_1}, m^{\epsilon_2}, d^{\epsilon_3}$ respectively, where $\epsilon_1, \epsilon_2, \epsilon_3 \in (0,1)$. However, for the purpose of our technical analysis, we need to establish boundary points.
For the analysis of lower bound, we set $R(\frac{n}{\log d})^{\frac{u}{2}}, R(\frac{n}{\log m})^{\frac{u}{2}}, R^{\frac{q}{u}}(\frac{n}{\log d})^{\frac{q}{2}}$ to lie in the intervals $[1,m], [1,m], [1,d]$, respectively. 
On the other hand, for  the analysis of upper bound, we set them to lie in the intervals $[\log m, \frac{m}{\log m}], [\log m, \frac{m}{\log m}], [\log d, \frac{m}{\log d}]$. 
Notably, this subtle mismatch does not affect the matching minimax risk established in our analysis.
\begin{theorem}\label{th4}
  Consider model \eqref{e1} under the double soft sparse structure defined in \eqref{lqq}, where $u,q \in (0, 1]$. Assume condition \eqref{lqqcondition} holds. We obtain the minimax rates as
	\begin{equation}\label{minimaxlqq}
	\mathcal{M}(\Theta_{u,q}^{m,d}(R)) =\inf_{\hat \theta} \sup_{\theta^* \in \Theta^{m,d}_{u,q}(R)} E\left[\|\hat\theta-\theta^*\|_F^2\right] \asymp \left\{\begin{array}{ll}
		R\left(\frac{n}{\log d}\right)^{\frac{u-2}{2}}+ R\left(\frac{n}{\log m}\right)^{\frac{u-2}{2}} & u > q\\
		R^{\frac{q}{u}}\left(\frac{n}{\log d}\right)^{\frac{q-2}{2}}+R\left(\frac{n}{\log m}\right)^{\frac{u-2}{2}} & u \le q,~ m > d\\
		R^{\frac{q}{u}}\left(\frac{n}{\log d}\right)^{\frac{q-2}{2}}  & u \le q,~ m \le d.
	\end{array}\right.
	\end{equation}
\end{theorem}
The optimal error rates in \eqref{minimaxlqq} exhibit different behaviors depending on the regimes.
Specifically, when $u > q$, the error rates are primarily influenced by the $\ell_u$ sparsity, while the $\ell_q$ sparsity does not significantly affect the error rates.
For $u \leq q$ and $m \geq d$, the error rates are jointly controlled by the $\ell_u$-ball and $\ell_q$-ball, representing a trade-off between these two risks.
On the other hand, when $m \leq d$, the error rate of the $\ell_u$-ball is dominated by the $\ell_q$-ball, resulting in a degeneration of the error rate as $R^{\frac{q}{u}}\left(\frac{n}{\log d}\right)^{\frac{q-2}{2}}$.
In this case, the error rate becomes independent of the number of groups $m$.

\section{Application to Linear Regression Model}\label{sec3}

In this section, we explore the double sparse linear regression model \cite{cai2019sparse, zhou2021high}, which directly applies the proposed sparse structure.
Consider a linear regression model with $n$ independent observations:
\begin{equation}\label{e12}
Y = X\beta^* + \xi = \sum_{j=1}^m X_{G_j}\beta^*_{G_j}+\xi.
\end{equation}
Here, $Y \in \mathbb{R}^{n}$ represents the response variable, $X\in \mathbb{R}^{n\times p}$ is the design matrix, and the random error term $\xi \in \mathbb{R}^n$ is a sub-Gaussian vector with parameter $\sigma^2$.
The parameter of interest, $\beta^* \in \mathbb{R}^{p}$, can be divided into $m$ predefined non-overlapping groups, each consisting of $d$ variables.
This implies that $p = m \times d$.
We denote the group structures as $\{G_j\}_{j=1}^m$ with $|G_j|=d$ for all $j \in [m]$.
Similar to definition \ref{df1}, we define the double sparse vector as follows:

\begin{definition}[Double sparse vector]\label{df3}
$\beta \in \mathbb{R}^{p}$ is called $(s, s_0)$-double sparse if it satisfies the following conditions:
\begin{equation}\label{double_vec}
\|\beta\|_{0,2} = \sum_{j=1}^m \mathbb{I}(\|\beta_{G_j}\|_2 \neq 0) \leq s\quad \mbox{and}\quad \beta_{G_j} \in \mathbb{B}^d_0(s_0)\ \text{for}\ j \in [m].
\end{equation}
\end{definition}
We denote the parameter spaces corresponding to \eqref{double_vec} as $\mathcal{B}^{m,d}_0(s, s_0)$.
In particular, the coefficient $\beta^*$ can be reshaped into a matrix $\theta^*$ as shown in \eqref{e1}.
Each group in $\beta^*$ corresponds to one column of $\theta^*$, i.e., $\beta^* = \mbox{Vec}(\theta^*)$.
Therefore, there is a close relationship between the parameter spaces corresponding to \eqref{e4} and \eqref{double_vec}.
Denote the support of $\beta \in \mathcal{B}^{m,d}_0(s, s_0)$ as $\mathcal{S}^{m,d}(s,s_0)$.
In specific, the support of $\beta\in \mathcal{B}^{m,d}_0(s, s_0)$ is characterized by supp$(\beta)=\{i \in [p]:\beta_i \neq 0, i \in [p]\}$.
Denote $S^*$ as supp$(\beta^*)$ and $G^*$ as $\{j \in [m]:\|\beta^*_{G_j}\|_{0,2} \neq 0, j \in [m]\}$.

It is evident that model \eqref{e1} is a special case of model \eqref{e12} when we set $n=p$ and $X = \sqrt{n}\mathbf{I}_n$.
Hence, the results from section \ref{sec2} can be extended with some additional conditions on the design matrix $X$.
In the remaining part of this section, we first extend the minimax rates for the estimation error \eqref{e9} to the double sparse linear regression model.
Subsequently, we develop a DSIHT algorithm and prove its optimality in the minimax sense.
\subsection{Estimation Error on $\ell_2$-loss}\label{subsec31}
We extend the sparse eigenvalue condition and restricted eigenvalue condition in \cite{raskutti2011minimax} to the double sparse linear regression, which is stated as the following assumption.

\begin{assumption}[Double Sparse Eigenvalue Condition]\label{ass1}
  For $\beta \in \mathbb{R}^p$, we assume 
  \begin{itemize}
    \item [(a)] Assume that there exists a positive constant $\tau_u < \infty$ such that
    \begin{equation*}
      \frac{1}{\sqrt{n}}\|X \beta\|_{2} \leq \tau_{u}\|\beta\|_{2}\ \ \text{for all}\ \ \beta \in \mathcal{B}^{m,d}_{0}(2 s, 2s_0).
    \end{equation*}
    \item [(b)] Assume that there exists a positive constant $\tau_{\ell}>0$ such that
    \begin{equation*}
      \frac{1}{\sqrt{n}}\|X \beta\|_{2} \geq \tau_{\ell}\|\beta\|_{2}\ \  \text{for all}\ \  \beta \in \mathcal{B}^{m,d}_{0}(2 s, 2s_0).
    \end{equation*}
  \end{itemize}
\end{assumption}

Assumption \ref{ass1} is a popular tool for analyzing high-dimensional linear regression.
It gives the range of the spectrum of the sub-matrices of $X$.
Another widely used assumption is the restricted isometry property (RIP, \cite{candes2005}).
RIP requires that the constants $\tau_u$ and $\tau_{\ell}$ are close to one.
In contrast, the constants in Assumption \ref{ass1} can be arbitrarily small and large, respectively.
In this regard, Assumption \ref{ass1} is a less stringent assumption compared to the RIP.

Combined with the technique used in Section 2, we directly obtain the results of the estimation error bounds for the high-dimensional double sparse linear regression.

\begin{corollary}\label{cor:1}
Consider model \eqref{e12} for a given design matrix $X \in \mathbb{R}^{n \times p}$.
Assume that Assumption 1 holds. Then, the minimax rate scales as
\begin{align}\label{e9}
  \begin{split}
    \inf_{\hat \beta} \sup_{\beta^* \in \mathcal{B}^{m,d}_0(s, s_0)}E\left[\|\hat \beta - \beta^*\|_2^2\right] \asymp
     \left\{\frac{\sigma^2}{n}( s \log \frac{em}{s}+ss_0\log \frac{ed}{s_0})\right\}.
  \end{split}
\end{align}
\end{corollary}

\subsection{Double Sparse Iterative Hard Thresholding Algorithm}\label{subsec32}

Refining this section, we introduce an extension to the traditional iterative hard thresholding (IHT) procedure as a constructive approach to addressing the problem of double sparse linear regression. 
To begin, we establish a new double sparse thresholding operator denoted as $\mathcal{T}_{\lambda,s,s_0}: \mathbb{R}^{p} \to \mathbb{R}^{p}$.

Subsequently, we proceed with the development of statistical guarantees for our proposed method, which serve to demonstrate its optimality in terms of convergence rate. 
The operator is achieved through a two-step process:

{\bf Step 1 (Element-wise thresholding)}: Define an element-wise hard thresholding operator $T_{\lambda}^{(1)}:\mathbb{R}^{p} \rightarrow \mathbb{R}^{p}$ as
$$
      \{\mathcal{T}^{(1)}_{\lambda}(\beta)\}_{j} = \beta_{j}\mathbb{I}(|\beta_{j}|\ge \lambda) , \quad \forall\ \beta \in \mathbb{R}^{p}.
$$

The operator $\mathcal{T}^{(1)}_{\lambda}$ refers to the traditional Iterative Hard Thresholding (IHT) operator, which preserves signals whose absolute magnitudes are greater than or equal to $\lambda$.

{\bf Step 2 (Group thresholding)}: Define the operator $\mathcal{T}^{(2)}_{\lambda,s,s_0}:\mathbb{R}^{p} \rightarrow \mathbb{R}^{p}$ as two parallel steps.
First of all, we reorder all elements in $\beta_{G_j} \in \mathbb{R}^d$ in descending order(of their absolute value) for all $j \in [m]$. After the following thresholding steps, the reordering is reversed back to the original natural order.

{\bf Inner-group thresholding}: Select the important entries of each group of $\beta$ and denote the index set as
\begin{equation}\label{eq:iht5}
    \mathcal{J}_{s} \coloneqq \{i \in [d]:\sum_{j=1}^m (\beta_{G_j})^2_i \ge s\lambda^2\},
\end{equation}  
where $(\beta_{G_j})_i$ represents the $i$th entry of $\beta_{G_j}$.

{\bf Outer-group thresholding}: 
Select the important groups of $\beta$ and denote the index set as
\begin{equation}\label{eq:iht6}
  \mathcal{J}_{s_0} \coloneqq \{j \in [m]:\|\beta_{G_j}\|_2^2 \ge s_0\lambda^2\}. 
\end{equation}

The Inner-group and Outer-group thresholding process apply a thresholding operation to filter out the entries within group and groups of $\beta$ that have magnitudes below a specific threshold, respectively.
From these two conditions, we define operator $\mathcal{T}^{(2)}_{\lambda, s, s_0}$ as
\begin{align*}
  (\{\mathcal{T}^{(2)}_{\lambda,s,s_0}(\beta)\}_{G_j})_i =
  \begin{cases}
    (\beta_{G_j})_i,\ &\text{if}\ i \in \mathcal{J}_{s}\ \mbox{and}\ j \in \mathcal{J}_{s_0}\\
    0, &\text{else}
  \end{cases}.
  \end{align*}

  We combine these two steps and define the operator $\mathcal{T}_{\lambda,s,s_0} = \mathcal{T}^{(2)}_{\lambda,s,s_0}\circ \mathcal{T}^{(1)}{\lambda}$.
  In simple terms, step 1 applies an element-wise hard thresholding operator, which retains the significant signals and removes the weak signals from $\beta$.
  Step 2 plays a crucial role in $\mathcal{T}_{\lambda,s,s_0}$. It uses a small rectangular region to further identify the important signal areas in $\beta$.
  For more details and motivation behind $\mathcal{T}_{\lambda,s,s_0}$, please refer to Section \ref{t3}.

  With the definition of the double sparse IHT operator $\mathcal{T}_{\lambda,s,s_0}$, 
we present the iterative hard thresholding algorithm for solving the double sparse linear regression as Algorithm \ref{alg:iht}.

\begin{algorithm}[htbp]
  \caption{\label{alg:iht}\textbf{D}ouble \textbf{S}parse \textbf{IHT} (DSIHT) algorithm}
  \begin{algorithmic}[1]
    \REQUIRE $X,\ Y,\ s,\ s_0,\ \kappa, \ \lambda_0,\ \lambda_{\infty}$.
    \STATE Initialize $t=0$ and $\beta^t = \bf{0}$.
    \WHILE {$\lambda_t \geq \lambda_{\infty},\ $}
    \STATE ${\beta}^{t+1} = \mathcal{T}_{\lambda_t,s,s_0}\left({\beta}^{t} + \frac{1}{n}X^\top (Y-X{\beta}^{t})\right)$.
    \STATE $\lambda_{t+1} = \sqrt{\kappa} \lambda_{t}$.
    \STATE $t = t+1$.
    \ENDWHILE
    \ENSURE $\hat \beta = \beta^{t-1}$.
  \end{algorithmic}
\end{algorithm}
In Algorithm \ref{alg:iht}, the contraction factor $0<\kappa<1$ serves as a step size, ensuring that the tuning sequence ${\lambda_t}$ decreases exponentially.
$\lambda_0$ is chosen as a suitably large value to ensure the sparsity of the estimator.
$\lambda_{\infty}$ denotes a threshold used to terminate the iterations of Algorithm \ref{alg:iht}.
A reasonable choice for $\lambda_{\infty}$ is approximately of the same magnitude as $\sqrt{\frac{\sigma^2\left(\frac{1}{s_0}\log(\frac{em}{s})+\log(\frac{ed}{s_0})\right)}{n}}$, as stated in the following theorem.

Given the sequence of threshold $\{\lambda_t\}_{t=0}^\infty$, we obtain the corresponding solution path $\{\beta^t\}_{t=1}^\infty$.
To conduct the theoretical analysis, we decompose the estimation error at each step into two parts:
\begin{align}\label{eq:H}
  \begin{split}
  H^{t+1} =&{\beta}^{t} + \frac{1}{n}X^\top (Y-X{\beta}^t)\\
   =& \beta^*+ \underbrace{(\frac{1}{n}X^\top X-\mathbf{I}_p)(\beta^* - {\beta}^t)}_{\text {Approximation error }}+ \underbrace{\frac{1}{n}X^\top \xi}_{\text {Model error }}.
  \end{split}
  \end{align}
To simplify the notations, let $\Phi = \frac{1}{n}X^\top X-\mathbf{I}_p$ and
$
\Xi = \frac{1}{n}X^\top \xi.
$
In what follows, we control the two sources of error, i.e., approximation error and model error, respectively in each iteration.

We first provide the Double Sparse Restricted Isometry Property ($\mbox{DSRIP}$) condition. 

\begin{definition}[DSRIP]\label{df2}
  Given integers $s \in [m]$ and $s_0\in [d]$, define $L_S$ and $U_S$ such that
  $$
  U_S = \max_{S\in\mathcal{S}^{m,d}(s,s_0)}\lambda_{\max}(X_S^\top X_S),
  $$
  $$
  \quad L_S = \min_{S\in\mathcal{S}^{m,d}(s,s_0)}\lambda_{\min}(X_S^\top X_S).
  $$
  Set $\delta_S = 1-\frac{L_S}{U_S}$.
  We say that the design matrix $X$ satisfies $\mbox{DSRIP}(s,s_0,c),0<c<1$ if
  $
  \delta_S \le c.
  $
\end{definition}

The traditional RIP condition is a valuable tool for conducting high-dimensional statistical analysis.
DSRIP extends the RIP condition to the double sparse structure, which is a less strict condition compared to RIP when considering double sparsity. Specifically, DRSIP requires nontrivial bounds on the sub-matrix $X^\top_S X_S$, where the sub-matrix is indexed by a $(s, s_0)$-shape set $S$.
In contrast, RIP imposes a stronger assumption, where the set $S$ consists of all index sets with $ss_0$ elements.

\begin{lemma}\label{lem:dsrip}
Assume that the design matrix $X$ satisfies DSRIP$(2s, 2s_0, \frac{\delta}{2})$. For any $S \in \mathcal{S}^{m, d}(2s, 2s_0)$, the operator norm of $\Phi_{SS}$ satisfies that $\|\Phi_{SS}\|_{op} \leq \delta$, where $\Phi_{SS}$ is the sub-matrix of $\Phi$ whose rows and columns are both listed in $S$.
\end{lemma}

Lemma \ref{lem:dsrip} demonstrates that $\Phi$ acts as a contraction factor for any $(2s, 2s_0)$-sparse vector.
Consequently, the DSRIP condition can compress the approximation error during the iteration procedure.
To control the model error, it is necessary to use the following lemma to capture the model complexity.
\begin{lemma}\label{lemma:iht1}
  Assume that $X$ satisfies DSRIP$(s, s_0, \frac{\delta}{2})$.
  For a constant $C>0$, the event
  \begin{align*}
    \mathcal{E} \coloneqq \left\{\forall S\in \mathcal{S}^{m,d}(s,s_0),  \sum_{i\in S}\Xi_{i}^2 \le \frac{10\sigma^2s\left(\log(\frac{em}{s})+s_0\log(\frac{ed}{s_0})\right)}{n} \right\}
  \end{align*}
  holds with probability greater than $1-\exp\left\{-C\left(s\log(\frac{em}{s})+ss_0\log(\frac{ed}{s_0})\right)\right\}$.
\end{lemma}

The good event $\mathcal{E}$ triggers the subsequent theoretical analysis. By satisfying $\mathcal{E}$, the operator $\mathcal{T}_{\lambda, s, s_0}$ gains control over the false positive variable within certain $(s, s_0)$-shaped sets.
To illustrate the role of the operator $\mathcal{T}_{\lambda, s, s_0}$, we consider model \eqref{e12} with $\beta^*=0$, which simplifies to a white noise model. We introduce a threshold $\lambda=\sigma\sqrt{\frac{10(\frac{1}{s_0}\log(\frac{em}{s})+\log(\frac{ed}{s_0}))}{n}}$ and aim to prove by contradiction that the selected set can be confined within a $(s, s_0)$-shape set.

To begin, we reshape the $p$-length vector into a $d \times m$ matrix and assume that the signal of each column decreases in absolute magnitude. Then, we apply the operator $\mathcal{T}_{\lambda,s_0}$ to obtain the selected set. We can analyze the properties of this operator by examining three distinct cases:

\begin{itemize}
\item[\textbf{Case 1:}] Assume that the set selected by $\mathcal{T}_{\lambda,s_0}$ covers no more than $s$ groups, but the number of entries in some groups exceeds $s_0$, as shown in Figure \ref{fig1}. The inner-group condition in Step 2 guarantees that the magnitude of the first $s$ rows is larger than $s_0 \lambda^2$. Consequently, the total magnitude in the red region is greater than $ss_0\lambda^2$, which contradicts the event $\mathcal{E}$ with high probability.

\item[\textbf{Case 2:}] Assume that the set selected by $\mathcal{T}_{\lambda,s_0}$ covers more than $s$ groups, but the number of entries in all groups is less than $s_0$, as shown in Figure \ref{fig2}. By the outer-group condition in Step 2, the magnitude of each selected column is larger than $s \lambda^2$. Similar to \textbf{Case 1}, the total magnitude in the red region exceeds $ss_0\lambda^2$, which contradicts the event $\mathcal{E}$ with high probability.

\item[\textbf{Case 3:}] Assume that the set selected by $\mathcal{T}_{\lambda,s_0}$ covers more than $s$ groups and more than $s_0$ non-zero entries in some groups. In Figure \ref{fig3}, considering Step 1 of the operator $\mathcal{T}_{\lambda,s_0}$, we guarantee that each entry in the selected set has a magnitude greater than $\lambda$. Consequently, in the groups ${G_1, G_2, G_3}$, the selected set located in the red region has $s_0$ non-zero entries in each group, implying that $G_1, G_2$, and $G_3$ have magnitudes greater than $s_0\lambda^2$ in the red region. Moreover, for $G_4$ and $G_5$, the magnitudes exceed $s_0 \lambda^2$ according to the outer-group condition in Step 2. By combining these two parts, we find that the total magnitude of the red region exceeds $ss_0\lambda^2$. This contradicts the event $\mathcal{E}$ with high probability.
\end{itemize}

By examining these cases, we have demonstrated that the selected set can be confined within a $(s, s_0)$-shape set in a white noise model, supporting the effectiveness of the operator $\mathcal{T}_{\lambda,s_0}$ in identifying the desired signal structure.
Overall, we have controlled the approximation error and model error, respectively, based on DSRIP condition and event $\mathcal{E}$.

\begin{figure}
  \begin{center}
    \begin{tikzpicture}[scale = 0.8]
      \draw[color=red!40,
      pattern={mylines[size= 5pt,line width=.8pt,angle=45]},
      pattern color=red] (0,2) rectangle (5,6);
      \draw[color=blue!40,
      pattern={mylines[size= 5pt,line width=.8pt,angle=-45]},
      pattern color=blue] (0,5) rectangle (4,6);
  \draw[color=blue!40,
      pattern={mylines[size= 5pt,line width=.8pt,angle=-45]},
      pattern color=blue] (0,1) rectangle (3,6);
      \draw[color=blue!40,
      pattern={mylines[size= 5pt,line width=.8pt,angle=-45]},
      pattern color=blue] (0,0) rectangle (2,1);
    \node[] at (2.5,6.7) {$s$};
    \node[] at (0.5,-0.5) {$G_1$};
    \node[] at (1.5,-0.5) {$G_2$};
    \node[] at (2.5,-0.5) {$G_3$};
    \node[] at (3.5,-0.5) {$G_4$};
    \node[] at (4.5,-0.5) {$G_5$};
    \node[] at (5.5,-0.5) {$G_6$};
    \node[] at (6.5,-0.5) {$G_7$};
    \node[] at (7.5,-0.5) {$G_8$};
    \node[] at (8.5,-0.5) {$G_9$};
    \node[] at (9.5,-0.5) {$G_{10}$};
    \draw[step=1,color=gray] (0,0) grid (10,6);
    \draw [thick, decorate, 
		decoration = {calligraphic brace, 
			raise=5pt, 
			aspect=0.5, 
			amplitude=4pt 
		}] (0,6) --  (5,6);
    \end{tikzpicture} 
    \end{center}
    \caption{Illustrative example of {\bf{case 1}}: Consider a scenario where we have 10 groups, each with an equal group size of $d=6$. To visualize the group structure, we reshape it into a $6\times 10$ matrix, where each column represents a group.
    In this example, we set $s=5$ and $s_0 = 4$. The blue region in the matrix represents the selected set, while the red region represents a $(s, s_0)$-shape subset with a total magnitude exceeding $ss_0 \lambda^2$.
    It's important to note that the entire vector of support is reshaped into a matrix, preserving the specific group structure.
  }\label{fig1}
  \end{figure}
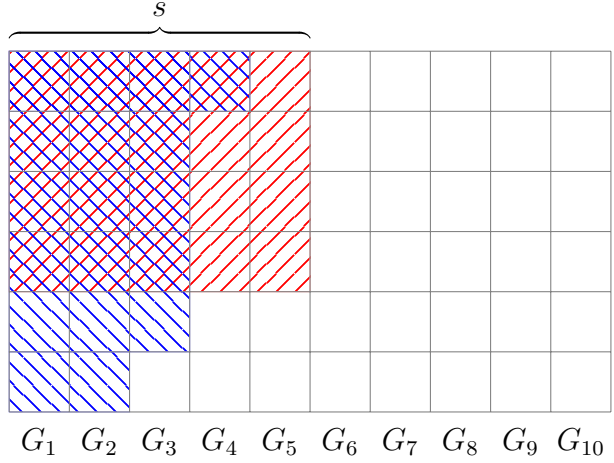

  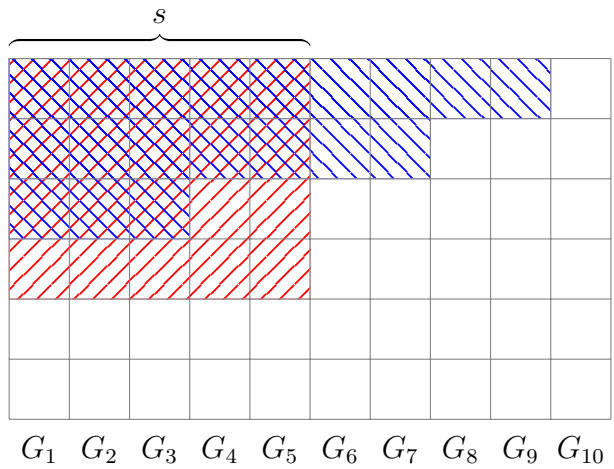
\begin{figure}
    \begin{center}
      \begin{tikzpicture}[scale = 0.8]
        \draw[color=red!40,
        pattern={mylines[size= 5pt,line width=.8pt,angle=45]},
        pattern color=red] (0,2) rectangle (5,6);
        \draw[color=blue!40,
        pattern={mylines[size= 5pt,line width=.8pt,angle=-45]},
        pattern color=blue] (0,5) rectangle (9,6);
       \draw[color=blue!40,
        pattern={mylines[size= 5pt,line width=.8pt,angle=-45]},
        pattern color=blue] (0,4) rectangle (7,6);
        \draw[color=blue!40,
        pattern={mylines[size= 5pt,line width=.8pt,angle=-45]},
        pattern color=blue] (0,3) rectangle (3,6);
      \node[] at (2.5,6.7) {$s$};
      \node[] at (0.5,-0.5) {$G_1$};
      \node[] at (1.5,-0.5) {$G_2$};
      \node[] at (2.5,-0.5) {$G_3$};
      \node[] at (3.5,-0.5) {$G_4$};
      \node[] at (4.5,-0.5) {$G_5$};
      \node[] at (5.5,-0.5) {$G_6$};
      \node[] at (6.5,-0.5) {$G_7$};
      \node[] at (7.5,-0.5) {$G_8$};
      \node[] at (8.5,-0.5) {$G_9$};
      \node[] at (9.5,-0.5) {$G_{10}$};
      \draw[step=1,color=gray] (0,0) grid (10,6);
      \draw [thick, decorate, 
      decoration = {calligraphic brace, 
        raise=5pt, 
        aspect=0.5, 
        amplitude=4pt 
      }] (0,6) --  (5,6);
      \end{tikzpicture} 
      \end{center}
      \caption{Illustrative example of {\bf{case 2}}.
      The elements in Figure \ref{fig2} are the same as  in Figure \ref{fig1}.
      }\label{fig2}
    \end{figure}

    \begin{figure}
      \begin{center}
        \begin{tikzpicture}[scale = 0.8]
          \draw[color=red!40,
          pattern={mylines[size= 5pt,line width=.8pt,angle=45]},
          pattern color=red] (0,2) rectangle (5,6);
          \draw[color=blue!40,
          pattern={mylines[size= 5pt,line width=.8pt,angle=-45]},
          pattern color=blue] (0,5) rectangle (9,6);
         \draw[color=blue!40,
          pattern={mylines[size= 5pt,line width=.8pt,angle=-45]},
          pattern color=blue] (0,4) rectangle (7,6);
          \draw[color=blue!40,
          pattern={mylines[size= 5pt,line width=.8pt,angle=-45]},
          pattern color=blue] (0,1) rectangle (3,6);
        \node[] at (2.5,6.7) {$s$};
        \node[] at (0.5,-0.5) {$G_1$};
        \node[] at (1.5,-0.5) {$G_2$};
        \node[] at (2.5,-0.5) {$G_3$};
        \node[] at (3.5,-0.5) {$G_4$};
        \node[] at (4.5,-0.5) {$G_5$};
        \node[] at (5.5,-0.5) {$G_6$};
        \node[] at (6.5,-0.5) {$G_7$};
        \node[] at (7.5,-0.5) {$G_8$};
        \node[] at (8.5,-0.5) {$G_9$};
        \node[] at (9.5,-0.5) {$G_{10}$};
        \draw[step=1,color=gray] (0,0) grid (10,6);
        \draw [thick, decorate, 
        decoration = {calligraphic brace, 
          raise=5pt, 
          aspect=0.5, 
          amplitude=4pt 
        }] (0,6) --  (5,6);
        \end{tikzpicture} 
        \end{center}
        \caption{Illustrative example of {\bf{case 3}}.
        The elements in Figure \ref{fig3} are the same as  in Figure \ref{fig1}.
        }\label{fig3}
      \end{figure}
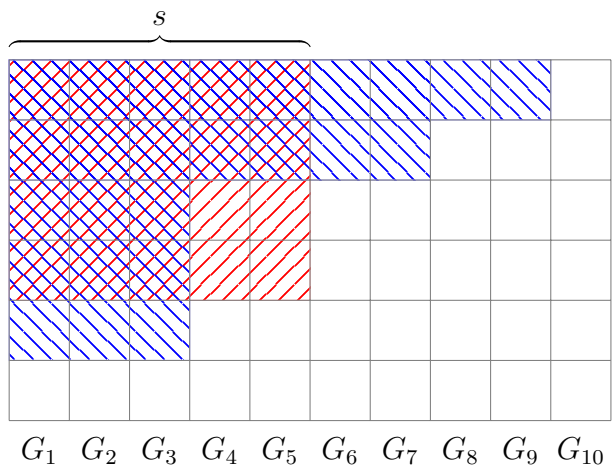

Based on Lemma \ref{lemma:iht1}, the following theorem shows that the corresponding estimator is $(2s, 2s_0)$-sparse.
In addition, the upper bound of estimation error decreases exponentially in each iteration.
\begin{theorem}\label{th3}
  Assume that $\beta^*$ is $(s,s_0)$-sparse and $X$ satisfies $\mbox{DSRIP}(2s, 2s_0,\delta/2)$. Let $\lambda_0>\lambda_{\infty}>0$ and $0<\kappa<1$. Assume that $\delta<\frac{1}{(3+\sqrt{2}+2\sqrt{3})^2}\wedge \kappa$, $\|\beta^* \|_2 \le \sqrt{ss_0}\lambda_0$ and 
  $$
  \lambda_{\infty} \ge \sqrt{\frac{40\sigma^2\left(\frac{1}{s_0}\log(\frac{em}{s})+\log(\frac{ed}{s_0})\right)}{n}}.
  $$
  We apply Algorithm \ref{alg:iht} and obtain a solution sequence $\{\beta^{t}\},t=1,2,\cdots$.
  Denote $S^t$ as the support of $\beta^t$.
 Then, with probability at least  $1-\exp\left\{-C\left(s\log(\frac{em}{s})+ss_0\log(\frac{ed}{s_0})\right)\right\}$, we have the following properties:
 \begin{itemize}
  \item Inside the true groups $G^*$, we have
 \begin{equation}\label{eq:iht11}
  (\mathop{\cup}_{j \in G^*} G_j) \cap S^t \cap (S^*)^c  \in \mathcal{S}^{m,d}(s,s_0).
\end{equation} 
\item Outside the true groups $G^*$, we have
\begin{equation}\label{eq:iht3}
  (\mathop{\cup}_{j \in G^*} G_j)^c \cap S^t   \in \mathcal{S}^{m,d}(s,s_0).
  \end{equation}
\item The upper bounds for the estimation error are
\begin{equation}\label{eq:iht4}
    \|\beta^t - \beta^* \|_2 \le \frac{3+\sqrt{2}+2\sqrt{3}}{2}\sqrt{ss_0}\lambda_t. 
  \end{equation}
\end{itemize}
\end{theorem}

Equation \eqref{eq:iht3} reveals that the false positive variables within the true groups $G^*$ can be effectively controlled within a $(s, s_0)$-shaped set. Additionally, \eqref{eq:iht3} demonstrates that our procedure can identify at most $s$ incorrect groups, while still maintaining control over the false positive variables outside the true groups within a $(s, s_0)$-shaped set.
These two results imply that the solution sequence ${\beta^t}$ remains $(2s, 2s_0)$-sparse in each iteration, highlighting the superior ability of our procedure to guarantee sparsity and minimize false positives. 

However, due to the nonconvex feature of IHT-style methods, the error $\|\beta^t - \beta^* \|_2$ cannot be guaranteed to decrease at each step.
To overcome this limitation, a common approach is to construct an upper bound that decreases at each step \citep{zhangtong2018, Zhu202014241}. In Algorithm \ref{alg:iht}, the sequence ${\lambda_t}$ exponentially decreases until it reaches the threshold $\lambda_{\infty}$. By choosing an appropriate value for $\lambda_{\infty}$, the upper bound given by \eqref{eq:iht4} achieves optimality in the minimax sense.
\begin{remark}
 Under the assumptions of Theorem \ref{th3}, Algorithm \ref{alg:iht} terminates after $O\left(\log(\frac{n\|\beta^*\|^2_2}{\sigma^2(s\log\frac{em}{s}+ss_0\log\frac{ed}{s_0})})/\log(\frac{1}{\kappa})\right)$ steps, which demonstrates that our method achieves optimal statistical accuracy with linear convergence.
\end{remark}

 \begin{remark}
  Assumption \ref{ass1} imposes a slightly weaker constraint on $X$ compared to DSRIP, as the latter requires all sub-matrices $X_{S}$ with $S \in \mathcal{S}^{m,d}(2s, 2s_0)$ to be near-isometries, meaning $\delta$ is close to one.
  However, the DSRIP condition is necessary for our theoretical analysis. We clarify this phenomenon in two aspects:
  \begin{itemize}
  \item The estimation method to obtain the minimax optimal estimator in Corollary \ref{cor:1} is NP-hard. However, we can find the optimal estimator as shown in Theorem \ref{th3} in polynomial time, more specifically, in linear time with high probability.
  In order to achieve these results, the analysis of our proposed IHT-style procedure requires a more stringent condition than Assumption \ref{ass1}.
  Similar constraints on $\delta$ for the analysis of the IHT procedure have been obtained in previous works \citep{huang2018constructive,zhangtong2018,Zhu202014241,zhang2023minimax}.
  \item Although the upper bound $\frac{1}{(3+\sqrt{2}+2\sqrt{3})^2}$ is much smaller than 1, it is not a necessary condition. Through our theoretical analysis, we can relax the constraint on the upper bound of $\delta$ by choosing a large constant factor in result \eqref{eq:iht4}.
  Many numerical results \cite{huang2018constructive,zhangtong2018,Zhu202014241,zhang2023minimax} demonstrate the promising performance of the IHT procedure in practice, even when the condition on $X$ is strict. 
  \end{itemize}
  \end{remark}

\subsection{Numerical experiments}\label{numerical}
In this section, we explore the numerical performance of our proposed DSIHT algorithm. We compare it with the traditional IHT algorithm \cite{ndaoud2020scaled} as well as its variant for group selection, which we refer to as IHT and GIHT, respectively.

The synthetic data sets are generated from the underlying model $y = X \beta^*+\xi$, where $\beta^* \in \mathbb{R}^{md}$ has $m$ groups with equal group size $d$.
We randomly draw design matrix $X \in \mathbb{R}^{n \times md}$ with i.i.d. standard normal entries.
The non-zero entries of $\beta^*$ are randomly chosen from standard normal distribution $\mathcal{N}(0, 1)$.
For the error term, $\xi_i$ is generated independently from $\mathcal{N}(0, \sigma^2)$.

Given an output $(\hat S, \hat \beta)$, we use the following metrics to assess the parameter estimation and variable selection:
\begin{itemize}
  \item \textbf{Estimation Error}: $\|\hat \beta - \beta^*\|_2$.
  \item \textbf{Mathew's Correlation Coefficient (MCC)}:
  \begin{align*}
    \text{MCC} = \frac{\text{TP}\times \text{TN}-\text{FP}\times \text{FN}}{\sqrt{(\text{TP+FP)(TP+FN)(TN+FP)(TN+FN)}}},
  \end{align*}
  where true positives (TP) and true negatives (TN) are defined as $\hat S \cap S^*$ and $\hat S^c \cap (S^*)^c$, respectively,
  and false positives (FP) and false negatives (FN) are defined as $\hat S \cap (S^*)^c$ and $\hat S^c \cap S^*$, respectively.
  Notably, a larger MCC means a better performance on variable selection.
\end{itemize}

We consider the following two simulation designs: 
\begin{itemize}
  \item [(1)] sample size $n$ varies from $300$ to $1000$ with increment equal to 50. The remaining model parameters are set as $m=50, d = 20, s = 2, s_0=10, \sigma=1$.
  \item [(2)] noise level $\sigma$ varies from $1$ to $5$ with increment equal to 0.5. We set $n=300$ and the remaining model parameters are the same as case (1).
\end{itemize}
All simulation results are based on 100 repetitions. 
Since GIHT allows all the variables in the selected groups into model, leading a much smaller MCC than DSIHT and IHT, we exclude the performance of GIHT on MCC.
The computational results are shown in Figure \ref{fig_n} and \ref{fig_sigma}.

\begin{figure}[htbp]
  \centering
  \includegraphics[scale = 0.58]{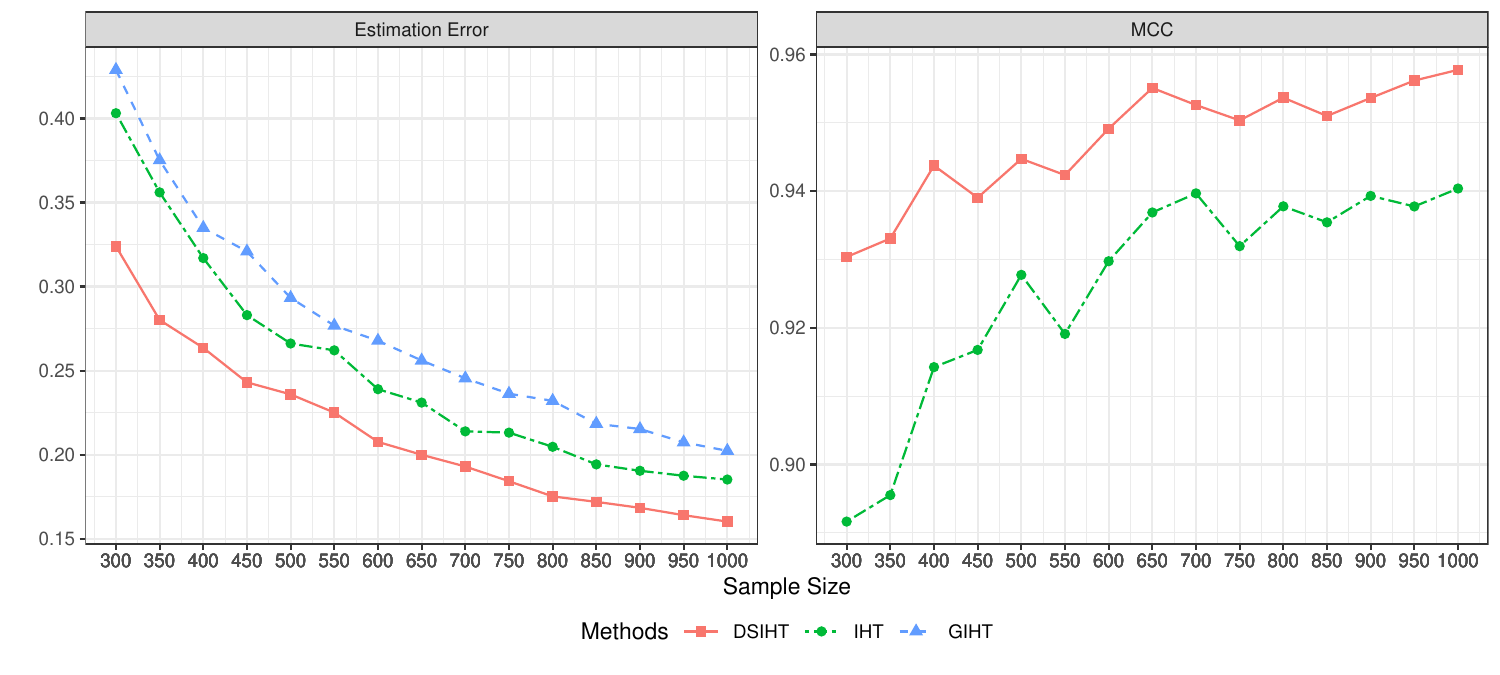}
  \caption{\label{fig_n} Performance measures as sample size $n$ increases.}
\end{figure}
As the sample size $n$ increases, it is evident from Figure \ref{fig_n} that all three methods demonstrate strong performance. However, it is worth highlighting that DSIHT exhibits a clear advantage over IHT and GIHT in terms of estimation error. This superiority of our method is particularly pronounced when dealing with a double sparse structure.
Furthermore, when it comes to variable selection, DSIHT significantly outperforms IHT. This outcome further solidifies the advantages of our proposed approach.

\begin{figure}[htbp]
  \centering
  \includegraphics[scale = 0.58]{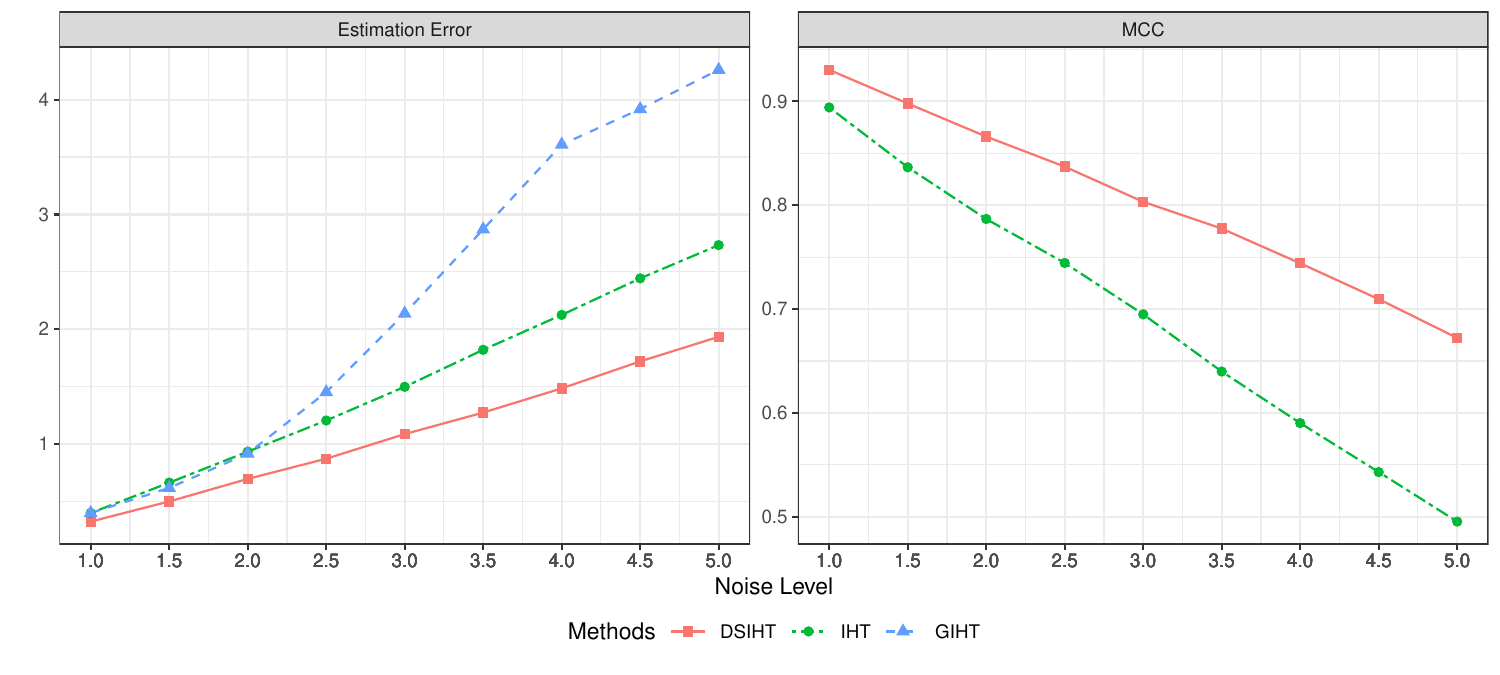}
  \caption{\label{fig_sigma} Performance measures as noise level $\sigma$ increases.}
\end{figure}

Figure \ref{fig_sigma} exhibits a similar trend as observed in Figure \ref{fig_n}. As the noise level $\sigma$ increases, the signal-to-noise ratio of the model diminishes, making it more challenging to accurately reconstruct the signal from the data.
At lower levels of $\sigma$, all three methods display comparable performance in terms of estimation error. However, as $\sigma$ increases, the differences in performance become significant, particularly for GIHT. The estimation error of GIHT shows an exponential growth as $\sigma$ increases. In contrast, both DSIHT and IHT demonstrate linear growth, but the slope of IHT is noticeably steeper than that of DSIHT.

Moreover, as $\sigma$ increases, the performance of both DSIHT and IHT deteriorates in terms of variable selection. However, DSIHT manages to maintain competitive or superior performance on variable selection even under high noise levels. This underscores the effectiveness and robustness of our method when dealing with double sparsity.

\section{Key Points of Proofs of Main Theorems}\label{sec4}
In this section, we briefly provide high-level proof sketches of the main technical results.
The detailed proofs are postponed to the appendix. 
\subsection{Proof Sketch of Theorem \ref{th1}}
The proofs of lower bounds generally follow an
information-theoretic method based on Fano's inequality \cite{thomas2006elements,yang1999information,yu1997assouad}. 
The technical details are slightly different between cases $q=0$ and $q \in (0,1]$.
Concretely speaking, for $q=0$, the proof is divided into two main steps:
\begin{itemize}
  \item [(1)] First, we construct the maximum packing set of the parameter space $\Theta^{m,d}_0(s, s_0)$ by coding theory \citep{thomas2006elements}. Our novel technique uses the combination of multi-ary and 2-ary Gilbert-Vashamov bound (Lemma \ref{lem1} in appendix) to establish the bounds of the packing number $M$ of $\Theta^{m,d}_0(s, s_0)$.  
  For two distinct elements $\theta^i, \theta^j$ in the packing set $\{\theta^1,\theta^2,\cdots,\theta^M\}$, the supports of $\theta^i$ and $\theta^j$ either share the same column index set or not. 
  We give more details about the construction of the packing set in Lemma \ref{lem2} in the appendix.
  \item [(2)] The next step is to derive a lower bound on $P(\psi \ne \tilde{\beta})$ by the Fano's inequality \citep{thomas2006elements}. We obtain
  \begin{equation}\label{fano}
    P(\psi \ne \tilde{\beta}) \ge 1-\frac{I(\bar Y;\psi)+\log 2}{\log M},
  \end{equation}
  where the estimator $\tilde{\beta}$ takes values in the packing set, $M$ is the packing number obtained in the first step, and $I(\bar Y;\psi)$ is the mutual information between random parameter $\psi$ in the packing set and observation $\bar Y$. We derive the upper bound of mutual information following the classical way in \cite{thomas2006elements}.  
\end{itemize}
On the other hand, for $q \in (0,1]$, \cite{raskutti2011minimax} derives the lower bound for $\ell_q$-ball by Yang-Barron Fano \citep{yang1999information}. In their framework, a sharp result of metric entropy of $\ell_q$-ball \cite{triebel2010fractals,kuhn2001lower} plays an essential role. 
In our work, we do not follow this idea because our metric entropy consists of an $\ell_0$ part and an $\ell_q$ part, which are incompatible in some sense. In other words, it is difficult to choose a pair of appropriate
covering and packing radii that matches the Yang-Barron Fano simultaneously. We use the following alternative techniques to overcome this difficulty:
\begin{itemize}
  \item [(1)] Inspired by \cite{gao2015rate}, when the lower bound consists of two parts, we can construct two corresponding parameter subspaces, respectively. Each of the parameter subspaces aims to specify the corresponding risk. Then, we combine these two parts by union bound.
  \item [(2)] To clarify the parameter subspace corresponding to the $\ell_q$ part of the risk, direct application of the metric entropy results is not straightforward. Inspired by the proofs in \cite{kuhn2001lower}, we use the techniques for lower bound of (dyadic) entropy number to construct our least favor distribution here, which is a substitution to techniques in \cite{raskutti2011minimax}.
\end{itemize}
\subsection{Proof Sketch of Theorem \ref{th2}}
  The proof of Theorem \ref{th2} mainly focuses on the direct analysis of the constrained least-squares estimators in \eqref{e7}.
It makes use of the classical upper bound technique for high-dimensional linear regression \cite{raskutti2011minimax}, which can be summarized as two main steps:
\begin{itemize}
  \item [(1)] From the definition of \eqref{e7}, we have $\|\bar Y-\hat \theta_q\|_F^2 \leq \|\bar Y - \theta^*\|_F^2.$
  By some simple algebras, we obtain the basic inequality
  \begin{equation}\label{e11}
    \|\hat\theta_q - \theta^*\|_F^2 \leq 2 |\langle \bar Z, \hat\theta_q-\theta^*\rangle| = \mbox{trace}\left(\bar Z^\top (\hat\theta_q-\theta^*)\right).
  \end{equation}
  Obviously, $\hat\theta_q - \theta^* \in  \Theta^{m,d}_q(2s, 2R_q)$.
  \item [(2)] This step uses the chaining and peeling techniques from empirical process theory \cite{geer2000empirical}, which helps us upper bound the right-hand side of \eqref{e11}.
      Moreover, some results of the covering number are required to derive the inequalities for the chaining results.
  More details can be found in Lemma \ref{lem2} in the appendix.
\end{itemize}

\subsection{Proof Sketch of Theorem \ref{th4}}\label{t4}

In the proof sketch, we outline the main steps of establishing the lower and upper bounds for the double sparse regression over $\ell_u^m(\ell_q^d)$ structure:

For the lower bound, we extend our techniques from the $\ell_0(\ell_0)$ structure to the $\ell_u(\ell_q)$ structure. The key challenge lies in dealing with the phase transition rates, as the relationships between $u$ and $q$, $m$ and $d$ are unknown.
\begin{itemize}
  \item  We transform the original problem into a linear programming formulation. For the $\ell_0(\ell_0)$ structure, where sparsity $(s,s_0)$ is known, we apply Fano's inequality to derive the lower bound. In the case of $\ell_u(\ell_q)$ structure, we choose $(s,s_0)$ appropriately later on.
  \item The minimax lower bound we seek depends on $(s,s_0)$, which is subject to constraints imposed by embedding properties and the high-dimensional condition \eqref{lqqcondition}. We maximize our objective function within the restricted domain to obtain lower bound rates in three scenarios, considering the relationships between $u$ and $q$, $m$ and $d$.
\end{itemize}

For the upper bound, the approach for obtaining the upper bound is similar to Theorem \ref{th2}. We utilize a generalized version of Sch{\"u}tt’s theorem for vector-valued spaces, as shown in \cite{edmunds2014schutt}. This allows us to calculate the covering number of the $\ell_u(\ell_q)$-norm unit ball.
The remaining steps of the proof involve combining the "chaining" and "peeling" techniques to establish the upper bound, leveraging the obtained covering number.

By following these steps, we establish both the lower and upper bounds for the double sparse regression over the $\ell_u(\ell_q)$ structure, providing insights into the phase transition rates and the optimality of our approach.
\subsection{Proof Sketch of Theorem \ref{th3}}\label{t3}

The proof of Theorem \ref{th3} relies on mathematical induction. We assume that \eqref{eq:iht11}, \eqref{eq:iht3}, and \eqref{eq:iht4} hold for $t$. Our goal is to show that \eqref{eq:iht11}, \eqref{eq:iht3}, and \eqref{eq:iht4} also hold for $t+1$.

We begin by proving \eqref{eq:iht11} and \eqref{eq:iht3} for $t+1$ through a contradiction argument. 
Suppose that \eqref{eq:iht11} fails for $t+1$, which means that $(\mathop{\cup}_{j \in G^*} G_j) \cap S^t \cap (S^*)^c \notin \mathcal{S}^{m,d}(s,s_0)$. 
In this case, there exists a $(s, s_0)$-shape set $S_{1, t+1}$ satisfying
 $$
  ss_0 \lambda_{t+1}^2 \le \sum_{i \in S_{1, t+1}}\{\mathcal{T}_{\lambda_{t+1}}(H^{t+1})\}^2_i
  $$
  as discussed in {\bf{Case 1}}. Similarly, if \eqref{eq:iht3} fails for $t+1$, implying that $(\mathop{\cup}_{j \in G^*} G_j)^c \cap S^t \notin \mathcal{S}^{m,d}(s,s_0)$, there exists a $(s, s_0)$-shape set $S_{2, t+1}$ satisfying
  $$
 ss_0 \lambda_{t+1}^2 \le \sum_{i \in S_{2, t+1}}\{\mathcal{T}_{\lambda_{t+1}}(H^{t+1})\}^2_i
 $$
 as discussed in {\bf{Case 1-3}}. Observe that for $j=1, 2$,
 \begin{align*}
  \sqrt{ss_0}\lambda_{t+1} \le \sqrt{\sum_{i \in S_{j, t+1}}\{\mathcal{T}_{\lambda_{t+1}}(H^{t+1})\}^2_i} \le \sqrt{\sum_{i \in S_{j, t+1}}\langle \Phi_{i}^\top , \beta^*-\beta^{t}\rangle^2} + \sqrt{\sum_{i\in S_{j, t+1}}\Xi_{i}^2}.
  \end{align*}
  By utilizing the DSRIP and Lemma \ref{lem1}, we can establish a contradiction, proving that \eqref{eq:iht11} and \eqref{eq:iht3} hold for $t+1$.
Given that \eqref{eq:iht11} and \eqref{eq:iht3} hold for $t+1$, we can directly demonstrate that \eqref{eq:iht4} holds for $t+1$ as well.
\section{Conclusion}\label{conclusion}

In this paper, we provide the comprehensive analysis of the statistical guarantees for the double sparse structure under hard and soft sparsity, specifically over $\ell_u(\ell_q)$-balls with $u,q \in [0,1]$. Our results fill an important gap in the understanding of the double sparse structure, as previous studies focused primarily on hard sparsity.

To establish the minimax lower bounds, we introduce a novel technique based on the Gilbert-Varshamov bounds, which enables concise and elegant proofs. This technique is not only applicable to the double sparse structure but also has broader implications for analyzing other related problems such as sparse additive models \cite{raskutti2012minimax, yuan2016minimax}. By combining this technique with Fano's inequality, we derive lower bounds for the estimation error.

On the other hand, we utilize the empirical process theory, specifically the chaining and peeling results, to obtain matching upper bounds. This demonstrates that the lower bounds we establish are tight. Interestingly, we discover a phase transition phenomenon in the minimax rates under double soft sparsity, revealing the intricate behavior of the double sparse structure.

Furthermore, we extend our results to high-dimensional double sparse linear regression. To effectively capture the double sparse structure, we develop a novel hard thresholding operator and propose the DSIHT algorithm. We demonstrate that our method achieves optimality in the minimax sense. However, we do not explore the fully adaptive version of our method in this paper, leaving it as a potential avenue for future research.

\appendix

\section{Appendix}
In the appendix, we provide additional lemmas and proofs that have been omitted from the main text.
To simplify the notations, we denote $C_1, C_2, C_3$ and other similar forms as some positive constants that can differ on different occurrences.  

\subsection{Technical Lemmas for lower bounds}
We first provide some Technical lemmas frequently used in the proof of our main theorems.

First, we present the useful results of Gilbert-Varshamov bound \cite{gilbert1952comparison}.
Denote $A_{Q}(n, d)$ as the maximum possible size of a $Q$-ary code $\mathcal{Q}$ with length $n$ and minimum Hamming distance $d$.
Then, we have
\begin{equation}\label{ap2}
  A_{Q}(n, d) \geq \frac{Q^{n}}{\sum_{j=0}^{d-1}{n \choose j}(Q-1)^{j}} .
\end{equation}
We use the 2-ary Gilbert-Varshamov bound on a Hamming distance sphere which is widely used in the high-dimensional problem. 

\begin{lemma}\label{lem1}
  Denote $S_k = \{x \in \{0, 1\}^m: \rho_H(x) = k\}$ as the Hamming ball with radius $k$. Then, the $\rho$-packing number of $S_k$ is 
  lower bounded as
  $$
  M(\rho;S_k, \|\cdot\|_H) \geq \frac{{m \choose k}}{\sum_{i=1}^{\rho}{m \choose i}}.
  $$
  Setting $\rho = C_1 k$, we obtain
  \begin{equation*}
    \log M(C_1 k;S_k, \|\cdot\|_H) \geq C_2 k\log \frac{em}{k}.
  \end{equation*}
\end{lemma}
The proof of Lemma \ref{lem1} can see \cite{wu2017lecture} for more details.
Based on Lemma \ref{lem1}, we obtain the following lemma.

\begin{lemma}[Lower bounds for the packing number]\label{lem2}
  The packing number of the parameter space $\Theta^{m,d}_0(s, s_0)$ with Hamming distance $ \frac{ss_0}{4}$ is lower bounded as
  \begin{equation*}
    M(\frac{ss_0}{4};\Theta^{m,d}_0(s, s_0), \|\cdot\|_H) \geq \exp(\frac{1}{4}s \log \frac{em}{s})\exp(\frac{s s_0}{4}\log \frac{ed}{s_0}).
  \end{equation*}
\end{lemma}
\begin{proof}
  The proof of Lemma \ref{lem2} contains the following four major steps.
  
  {\bf Step 1}:
  Referring to Lemma \ref{lem1}, we can find a packing set $\widetilde{\Gamma} \subset \Gamma$ that meets the following conditions:
  \begin{itemize}
    \item The cardinality of $\widetilde{\Gamma}$ satisfies $|\widetilde{\Gamma}| \geq \exp\left(\frac{s}{4}\log\frac{em}{s}\right)$.
    \item For any $\gamma^1$ and $\gamma^2$ in $\widetilde{\Gamma}$, the Hamming distance $\rho_H(\gamma^1, \gamma^2) \geq \frac{s}{4}$.
\end{itemize}

  {\bf Step 2}:
  Referring to Lemma \ref{lem1}, we can find a packing set $ \widetilde{B} \subset B$ that meets the following conditions:
  \begin{itemize}
    \item The cardinality of $\widetilde{B}$ satisfies $|\widetilde{B}| \geq \exp(\frac{s_0}{2}\log\frac{ed}{s_0})$.
    \item For any $b^1$ and $b^2$ in $\widetilde{B}$, the Hamming distance $\rho_H(b^1, b^2) \geq \frac{s_0}{2}$.
\end{itemize}

  {\bf Step 3}: 
  We consider the maximum possible size of a $|\widetilde{B}|$-ary code with length $s$. From \eqref{ap2}, we have
  \begin{equation*}
    A_{|\widetilde{B}|}(s, \frac{s}{2}) \geq \frac{{|\widetilde{B}|}^s}{\sum_{i=0}^{\frac{s}{2}-1} {s \choose i}(|\widetilde{B}|-1)^i} \geq \frac{{|\widetilde{B}|}^{\frac{s}{2}}}{\exp(s)} \geq \frac{\exp(\frac{ss_0}{4}\log \frac{ed}{s_0})}{\exp(s)},
  \end{equation*}
      where the second inequality follows from $\sum_{i=0}^{\frac{s}{2}-1}{s \choose i} \leq \exp(s)$ and the third inequality follows from step 2.
      Consequently, we obtain a $\frac{s}{2}$-packing set $\mathcal{Q}$ for the $|\widetilde{B}|$-ary code with length $s$. In other words, for any $Q^1$ and $Q^2$ in $\mathcal{Q}$, we have the Hamming distance $\rho_H(Q^1, Q^2) \geq \frac{s}{2}$.     

      Steps 1-3 can be understood as a procedure to construct a packing set $\widetilde{\Theta}^{m,d}_0(s, s_0) \subset \Theta^{m,d}_0(s, s_0)$.
      Intuitively, the elements of $\widetilde{\Gamma}$ determine the positions of the non-zero columns in $\theta \in \widetilde{\Theta}^{m,d}_0(s, s_0)$, while $\widetilde{B}$ specifies the possible candidates for each non-zero column.
      After determining the locations of the non-zero columns, we utilize a code $Q$ from $\mathcal{Q}$ in step 3 to determine the choices for these columns. Subsequently, based on $Q$, we select the corresponding columns from $\widetilde{B}$ to populate the matrix $\theta$
  
  {\bf Step 4}: 
  In step 4, we demonstrate that $\widetilde{\Theta}^{m,d}_0(s, s_0)$ is a $\frac{s s_0}{4}$-packing set of $\Theta^{m,d}_0(s, s_0)$ and derive a lower bound for the cardinality of $\widetilde{\Theta}^{m,d}_0(s, s_0)$.
  We analyze the following two cases for $\theta^i \neq \theta^j \in \widetilde{\Theta}^{m,d}_0(s, s_0)$:
  
  Case 1: If $\gamma^i \neq \gamma^j$, where $\gamma^i$ and $\gamma^j$ represent the indices of the non-zero columns in $\theta^i$ and $\theta^j$ respectively (as obtained from $\widetilde{\Gamma}$ in step 1), we have $\rho_H(\gamma^i, \gamma^j) \geq \frac{s}{4}$. Since both $\theta^i$ and $\theta^j$ belong to $\mathbb{B}^d_0(s_0)$, it follows that $\rho_H(\theta^i, \theta^j) \geq \frac{ss_0}{4}$.
  
  Case 2: If $\gamma^i = \gamma^j$, from step 3, we know that $\theta^i$ and $\theta^j$ have at least $\frac{s}{2}$ different columns, as determined by $\mathcal{Q}$. Moreover, step 2 indicates that the Hamming distance between different columns in $\widetilde{B}$ is at least $\frac{s_0}{2}$. Thus, we have $\rho_H(\theta^i, \theta^j) \geq \frac{ss_0}{4}$.
  
  Combining steps 1-3, we can conclude:
  \begin{align*}
  M(\frac{ss_0}{4};\Theta^{m,d}_0(s, s_0), \|\cdot\|_H) \geq |\widetilde{\Theta}^{m,d}_0(s, s_0)| \geq \exp\left(\frac{1}{4}s \log \frac{em}{s}\right)\exp\left(\frac{ss_0}{4}\log \frac{ed}{s_0}\right).
  \end{align*}
  Note that $\exp(s)$ is negligible due to its slower growth rate compared to the first term.
  \QEDB
\end{proof}

Next, we prove the upper bounds for the covering number of the parameter spaces.

\begin{lemma}[Upper bounds for the covering number]\label{lem3}
  \quad
\begin{itemize}
    \item [(a)] For $q=0$, 
    \begin{equation*}
      \log N(\varepsilon;\Theta^{m,d}_0(s, s_0), \|\cdot\|_F) \leq s\log \frac{em}{s} + ss_0 \log \frac{ed}{s_0} + 2s\log\frac{1}{\varepsilon}.
    \end{equation*}
    \item [(b)] For $q \in (0, 1]$ and for all $\varepsilon \in [\sqrt{s}C_q R_q^{\frac{1}{q}}(\frac{\log d}{d})^{\frac{2-q}{2q}}, \sqrt{s}R_q^{\frac{1}{q}}]$, 
    \begin{equation*}
      \log N(\varepsilon;\Theta^{m,d}_q(s, R_q), \|\cdot\|_F) \leq s\log \frac{em}{s} + s(C_q \frac{s R_q^{\frac{2}{q}}}{\varepsilon^2})^{\frac{q}{2-q}} \log d.
    \end{equation*}
  \end{itemize}
\end{lemma}

\begin{proof}
We first prove case (b), and case (a) can be proved similarly. For case (b), define $\widetilde{B}^d_q(R_q)$ as the $\frac{\varepsilon}{\sqrt{s}}$-covering set of $B^d_q(R_q)$ with respect to $\|\cdot\|_2$.
Given $\theta \in \Theta^{m,d}_q(s, R_q)$, the $j$-th column $\theta_j \in B^d_q(R_q)$. We can find a vector $b^j \in\widetilde{B}^d_q(R_q)$ such that $\|b^j - \theta_j\|_2^2 \leq \frac{\varepsilon^2}{s}$.
Consequently, we have 
\begin{equation*}
  N(\varepsilon;\Theta^{m,d}_q(s, R_q), \|\cdot\|_F) \leq {m \choose s} \left(N(\frac{\varepsilon}{\sqrt{s}};B_q^d(R_q), \|\cdot\|_2)\right)^s.
\end{equation*}
By inverting known results on entropy numbers of $\ell_q$-balls, i.e., \eqref{high}, we have
\begin{align*}
    \epsilon_k(B_q^d(1)) = \frac{\varepsilon}{\sqrt{s}} \leq &\left(\frac{\log(\frac{d}{k}+1)}{k}\right)^{\frac{1}{q}-\frac{1}{2}}\\
  \leq & \left(\frac{\log d}{k}\right)^{\frac{1}{q}-\frac{1}{2}}.
\end{align*}
Inverting this inequality for $k = \log N(\frac{\varepsilon}{\sqrt{s}}; B_q^d(1), \|\cdot\|_2)$ and allowing for a ball radius $R_q^{1/q}$ yields
\begin{align*}
  \log N(\frac{\varepsilon}{\sqrt{s}}; B_q^d(1), \|\cdot\|_2) \leq (C_q \frac{s R_q^{\frac{2}{q}}}{\varepsilon^2})^{\frac{q}{2-q}} \log d.
\end{align*}
  The conditions on the range of $\varepsilon$ guarantee that $k \in [\log d, d]$.
Combining these results, we have
\begin{align*}
  \log N(\varepsilon;\Theta^{m,d}_q(s, R_q), \|\cdot\|_F) &\leq s\log \frac{em}{s} + s\log N(\frac{\varepsilon}{\sqrt{s}}; B_q^d(1), \|\cdot\|_2)\\
  &\leq s\log \frac{em}{s} + s(C_q \frac{s R_q^{\frac{2}{q}}}{\varepsilon^2})^{\frac{q}{2-q}} \log d.
\end{align*}
Case (a) can be derived in a similar way as case (b), in which we just need to replace the covering number of $B_q^d(R_q)$ by $B_0^d(s_0)$.
\QEDB
\end{proof}

\subsection{Proof of Theorem 1}\label{proof1}

\subsubsection{Proof of case (a)}
\begin{proof}\label{proof_lower}
  For case (a), let's consider the $\frac{ss_0}{4}$-packing set $\widetilde{\Theta}^{m,d}_0(s, s_0) = {\theta^1, \ldots, \theta^{M}}$ defined in Lemma \ref{lem2}, where $M$ is the cardinality of $\widetilde{\Theta}^{m,d}_0(s, s_0)$.

  We set all non-zero elements of $\theta \in \widetilde{\Theta}^{m,d}_0(s, s_0)$ equal to 1, and let $\vartheta^i = \theta^i \delta$, where $\delta$ is a parameter to be determined below.
  For any $\vartheta^i \neq \vartheta^j$, since each $\theta \in \widetilde{\Theta}^{m,d}_0(s, s_0)$ has at most $ss_0$ non-zero elements, we have:
  \begin{equation}\label{ap16}
  \|\vartheta^i-\vartheta^j\|_F^2 \leq 2ss_0\delta^2,\quad \forall i,j \in [M].
  \end{equation}
  On the other hand, from the construction of $\widetilde{\Theta}^{m,d}_0(s, s_0)$, we have:
  \begin{equation}\label{ap3}
  \|\vartheta^i-\vartheta^j\|_F^2 \geq \frac{1}{4}ss_0\delta^2,\quad \forall i,j \in [M].
  \end{equation}
  Applying the property of mutual information \cite{wu2017lecture}, we obtain an upper bound on $I(y;B)$:
  \begin{align}\label{ap6}
  \begin{split}
  I(\bar Y;\psi) &\leq \frac{1}{{M \choose 2}} \sum_{i \neq j} KL(\vartheta^i \| \vartheta^j)\\
  &= \frac{1}{{M \choose 2}} \sum_{i \neq j} \frac{n}{2\sigma^2}\|\vartheta^i-\vartheta^j\|_F^2\\
  &\leq \frac{n}{\sigma^2}ss_0 \delta^2,
  \end{split}
  \end{align}
  where $KL(\cdot \| \cdot)$ denotes the Kullback-Leibler divergence, and the last inequality follows from \eqref{ap16}.
  Combining Fano's inequality \cite{thomas2006elements} and \eqref{ap6}, we have:
  \begin{equation*}
  P(\hat \vartheta \neq \psi ) \geq 1 - \frac{\frac{n}{\sigma^2}ss_0 \delta^2+\log 2}{\log M},
  \end{equation*}
  where $\psi$ is the random vector uniformly distributed over the packing set $\widetilde{\Theta}^{m,d}_0(s, s_0)$.
  To ensure that $P(\hat \vartheta \neq \psi) \geq \frac{1}{2}$, it suffices to choose:
  \begin{equation*}
  \delta = \frac{1}{2}\sqrt{\left(\frac{1}{4} s \log \frac{em}{s}+\frac{ss_0}{4}\log \frac{ed}{s_0}\right)\frac{\sigma^2}{ss_0n}}.
  \end{equation*}
  Substituting this into equation \eqref{ap3} and using Lemma \ref{lem2}, we have:
  \begin{equation*}
  \inf_{\hat \theta} \sup_{\theta \in \Theta^{m,d}_0(s, s_0)}P\left(\|\hat{\theta}-\theta\|_F^2 \geq \frac{\sigma^2}{64n}\left(s \log \frac{em}{s}+ss_0\log \frac{ed}{s_0}\right)\right) \geq \frac{1}{2},
  \end{equation*}
  which completes the proof. A Markov’s inequality argument leads to the lower bound in expectation.
\end{proof}

\subsubsection{Proof of case(b)}
\begin{proof}
	Here we consider splitting the proof of case (b) into two cases.
	The first case focuses on the parameter space in that the element has the same entries with different indices of the non-zero columns.
	The other case studies the parameter space that the element has the same indices of the non-zero columns but the entries are different.
	At last, we combine these two cases by union bounds.
	
	Define $\Lambda = \{\lambda : \lambda \in \delta \times \{0, 1\}^d\ \text{and}\ \|\lambda\|_0 \leq s_0\}$, where $\delta$ satisfies that $\delta^q s_0=R_q$, i.e, $\Lambda \subseteq \mathbb{B}^d_q(R_q)$.

{\bf Case 1:} Fix $\lambda \in \Lambda$ and produce the index set $\tilde{\Gamma}$  according to the step 1 of Lemma \ref{lem2}. And we obtain:
$$
\widetilde{\Theta}_q^{m,d}(s, R_q, \lambda) = \{\gamma \otimes \lambda: \gamma \in \tilde{\Gamma}\}
$$

Therefore, $|\widetilde{\Theta}_q^{m,d}(s, R_q, \lambda)| = |\tilde{\Gamma}|$. For any $\theta^i\neq \theta^j \in \widetilde{\Theta}_q^{m,d}(s, R_q, \lambda)$, because we have $\rho_H(\gamma^i, \gamma^j) \geq \frac{s}{4}$, therefore
\begin{equation}\label{ap17}
  \|\theta^i - \theta^j\|_F^2 \geq \frac{1}{4}ss_0\delta^2,
\end{equation}
which implies that $\widetilde{\Theta}_q^{m,d}(s, R_q, \lambda)$ is a $\frac{1}{4}s s_0 \delta^2$-packing set of $\Theta_q^{m,d}(s, R_q)$.
On the other hand, $\gamma^i$ and $\gamma^j$ differ in at most $2s$ positions. Therefore, 
\begin{equation}\label{ap18}
  \|\theta^i - \theta^j\|_F^2 \leq 2ss_0\delta^2.
\end{equation}
Similar to \eqref{ap6}, we have
\begin{align}\label{ap13}
  \begin{split}
    I(\bar Y;\psi) \leq& \frac{1}{{M_{\lambda} \choose 2}} \sum_{i \neq j} KL(\theta^i || \theta^j)\\
    = & \frac{1}{{M_{\lambda}  \choose 2}} \sum_{i \neq j} \frac{n}{2\sigma^2}\|\theta^i-\theta^j\|_F^2\\
    \leq&\frac{n}{\sigma^2}ss_0 \delta^2,
  \end{split}
\end{align}
where the last inequality follows from \eqref{ap18}.
Combining Fano's inequality and \eqref{ap13}, we have
\begin{equation*}
  P(\psi \neq \hat \theta) \geq 1 - \frac{\frac{n}{\sigma^2}ss_0 \delta^2+\log 2}{\frac{1}{4}s\log \frac{em}{s}},
\end{equation*}
where $\psi$ is the random vector uniformly distributed over the packing set $\widetilde{\Theta}_q^{m,d}(s, R_q, \lambda)$.
It suffices to choose $\delta = \sqrt{\frac{\sigma^2}{4ns_0}\log \frac{em}{s}}$ to guarantee $P(\psi \neq \hat \theta) \geq \frac{1}{2}$.
Note that the triple $(\delta, s_0, R_q)$ satisfies $\delta^q s_0=R_q$. Some algebras imply that $\delta = (\frac{\sigma^2\log \frac{em}{s}}{16nR_q})^{\frac{1}{2-q}}$.
Substituting into \eqref{ap17}, we have 
\begin{equation}\label{ap4}
  \inf_{\hat \theta} \sup_{\theta \in \widetilde{\Theta}_q^{m,d}(s, R_q, \lambda)}P\left( \|\hat \theta - \theta \|^2_F \geq \frac{\sigma^2}{64n}s \log \frac{em}{s}\right) \geq \frac{1}{2}.
\end{equation}

{\bf Case 2:} Consider an index vector $\gamma \in \widetilde{\Gamma}$ such that $\gamma_i = 1$ for $i \in [s]$ and $\gamma_i = 0$ otherwise. Similar to steps 2-3 in Lemma \ref{lem2}, let $\tilde{\Lambda} \subseteq \Lambda^s$ denote the set consisting of all $s \times d$ matrices obtained by those steps. In other words,
\begin{equation*}
\widetilde{\Theta}_q^{m,d}(s, R_q, \gamma) = \{\theta: \theta_i = L_i,\ i \in [s]; \theta_{i} = 0,\ \text{otherwise};\ L \in \tilde{\Lambda}^s\}.
\end{equation*}
From the {\bf{Case} 2} in Step 4 in Lemma \ref{lem1}, we have the following inequality:
\begin{equation*}
\|\theta^i - \theta^j \|_F^2 \geq \frac{1}{4}s s_0 \delta^2,
\end{equation*}
which implies that $\widetilde{\Theta}_q^{m,d}(s, R_q, \gamma)$ forms a $\frac{1}{4}s s_0 \delta^2$-packing set of $\Theta_q^{m,d}(s, R_q)$.
Additionally, we have:
\begin{equation*}
|\widetilde{\Theta}_q^{m,d}(s, R_q, \gamma)| \geq \frac{\exp(\frac{ss_0}{4}\log \frac{ed}{s_0})}{\exp(s)}.
\end{equation*}

On the other hand, the following inequality holds:
\begin{equation*}
\|\theta^i - \theta^j \|_F^2 \leq 2s s_0 \delta^2.
\end{equation*}

By combining Fano's inequality and a similar derivation as in \eqref{ap13}, we obtain:
\begin{equation*}
P(\hat{\theta} \neq \psi) \geq 1 - \frac{\frac{n}{\sigma^2}ss_0 \delta^2+\log 2}{\frac{1}{4}ss_0\log \frac{ed}{s_0}},
\end{equation*}
where $\psi$ is a random vector uniformly distributed over the packing set $\widetilde{\Theta}_q^{m,d}(s, R_q, \gamma)$.
To ensure $P(\hat{\theta} \neq \psi) \geq \frac{7}{8}$, let's assume $s_0 = d^v$ for some constant $v \in (0, 1]$. Choosing $\delta = (\frac{(1-v)\sigma^2}{4n}\log d)^{\frac{1}{2}}$ based on \eqref{eq:Rq}, we obtain $s_0 = R_q(\frac{(1-v)\sigma^2}{4n}\log d)^{-\frac{q}{2}}$. As a result, we have:
\begin{align}\label{ap5}
\begin{split}
\inf_{\hat{\theta}} \sup_{\theta \in \widetilde{\Theta}_q^{m,d}(s, R_q, \gamma)}P\left( \|\hat{\theta} - \theta \|^2_F \geq\frac{1}{4n}s R_q(\frac{(1-v)\sigma^2}{4n}\log d)^{1-\frac{q}{2}}\right)\geq \frac{7}{8}.
\end{split}
\end{align}
By applying the union bound to combine the risk in \eqref{ap4} and \eqref{ap5}, for any $\hat{\theta}$, we have:
\begin{align*}
  P(\|\hat \theta - \theta\|_F^2 &\geq \frac{1}{4}s R_q(\frac{\sigma^2v }{4n}\log d)^{1-\frac{q}{2}}+\frac{\sigma^2}{64n}s\log\frac{em}{s})\\
  &\geq1-P(\|\hat \theta - \theta\|_F^2 \leq \frac{1}{4}s R_q(\frac{(1-v)\sigma^2}{4n}\log d)^{1-\frac{q}{2}})\\
  &\quad -P(\|\hat \theta - \theta\|_F^2 \leq \frac{\sigma^2}{64n}s\log\frac{em}{s})\\
  &=P(\|\hat \theta - \theta\|_F^2 \geq \frac{1}{4}s R_q(\frac{(1-v)\sigma^2}{4n}\log d)^{1-\frac{q}{2}})\\
  &\quad +P(\|\hat \theta - \theta\|_F^2 \geq \frac{\sigma^2}{64n}s\log\frac{em}{s})-1.
\end{align*}
Taking the supremum on both sides, we obtain:
\begin{align*}
  &\sup_{\theta \in \Theta_q^{m,d}(s, R_q)}P(\|\hat \theta - \theta\|_F^2 \geq \frac{1}{4}s R_q(\frac{\sigma^2v }{4n}\log d)^{1-\frac{q}{2}}+\frac{\sigma^2}{64n}s\log\frac{em}{s})\\
  \geq&\sup_{\theta \in \widetilde{\Theta}_q^{m,d}(s, R_q, \lambda)}P(\|\hat \theta - \theta\|_F^2 \geq \frac{1}{4}s R_q(\frac{(1-v)\sigma^2}{4n}\log d)^{1-\frac{q}{2}})\\
  &\quad +\sup_{\theta \in \widetilde{\Theta}_q^{m,d}(s, R_q, \gamma)}P(\|\hat \theta - \theta\|_F^2 \geq \frac{\sigma^2}{64n}s\log\frac{em}{s})-1\\
  \geq&\frac{7}{8}+\frac{1}{2}-1 = \frac{3}{8}.
\end{align*}
Therefore,
\begin{align*}
\inf_{\hat{\theta}} \sup_{\theta \in \Theta_q^{m,d}(s, R_q)}P\left(\|\hat{\theta}-\theta\|_F^2 \geq \frac{\sigma^2}{64n}s\log\frac{em}{s}+\frac{1}{4}s R_q(\frac{(1-v)\sigma^2}{4n}\log d)^{1-\frac{q}{2}}\right)\geq \frac{3}{8},
\end{align*}
which completes the proof of \eqref{e6}. A Markov’s inequality argument leads to the lower bound in expectation.
\QEDB
\end{proof}

\subsection{Technical Lemmas for upper bounds}

Before formally providing the proof of upper bounds, we provide some useful technical lemmas.
Define the function $f(v ; X)$, where $v \in \mathbb{R}^{d \times m}$ is the vector to be optimized over, and $X$ is some random matrix. 
We are interested in the constrained problem $\sup _{\rho(v) \leq r, v \in A} f(v ; X)$, where $\rho: \mathbb{R}^{d \times m} \rightarrow \mathbb{R}^{+}$ is some increasing constraint function, and $A$ is a non-empty set. 
The goal is to bound the probability of the event defined by
$$
\mathcal{Z}:=\left\{X \in \mathbb{R}^{n \times d} : \exists v \in A \text { s.t. } f(v ; X) \geq 2 g(\rho(v))\right\}
$$
where $g: \mathbb{R} \rightarrow \mathbb{R}^+$ is some strictly increasing function.
\begin{lemma}[Peeling, Lemma 9 of \cite{raskutti2011minimax}]\label{peeling}
Suppose $g(r) \geq \mu,\ \forall r \geq 0$. There exists some constant $c>0$ such that for all $r>0$, we have the tail bound
$$
\mathbb{P}\left\{\sup _{v \in A, \rho(v) \leq r} f(v ; X) \geq g(r)\right\} \leq 2 \exp \left(-c a_{n} g(r)\right)
$$
for some $a_{n}>0$. Then, we have
$$
\mathbb{P}(\mathcal{Z}) \leq \frac{2 \exp \left(-4 c a_{n} \mu\right)}{1-\exp \left(-4 c a_{n} \mu\right)}.
$$
\end{lemma}

The proof of Lemma \ref{peeling} can be seen \cite{raskutti2011minimax} for more details.
In addition,  we also need the double sparse version of Lemma 6 of \cite{raskutti2011minimax}.
Denote $\widetilde{\mathcal{S}}(\Theta^{m,d}_0(2s, 2s_0), r) = \Theta^{m,d}_0(2s, 2s_0) \cap \{\theta \in \mathbb{R}^{d \times m}:\|\theta\|_F^2 \leq r\}$
and $\widetilde{\mathcal{S}}(\Theta^{m,d}_q(2s, 2R_q), r) = \Theta^{m,d}_q(2s, 2R_q) \cap \{\theta \in \mathbb{R}^{d \times m}|\|\theta\|_F^2 \leq r\}$.
\begin{lemma}\label{covering}
There exists some constants $C_1,C_2>0$ such that for any $r > 0$, we have
\begin{align*}
  \sup_{\theta \in \widetilde{\mathcal{S}}(\Theta^{m,d}_0(2s, 2s_0), r)} |\langle \bar Z, \theta \rangle| \leq C_u\sigma r \sqrt{\frac{1}{n}(s \log \frac{em}{s}+ss_0\log \frac{ed}{s_0})}
\end{align*}
with probability greater than $1-C_1\exp \{-C_2 (s \log \frac{em}{s}+ss_0\log \frac{ed}{s_0})\}$.
\end{lemma}

Lemma \ref{covering} is a direct consequence of Lemma 6 of \cite{raskutti2011minimax}, in which we replace the covering number of $\ell_0$-ball by 
the covering number of $\Theta^{m,d}_{0}(2 s, 2s_0)$ in Lemma \ref{lem3} (a).

\begin{lemma}\label{chaining}
  Assume that
   \begin{equation}\label{ap15}
     \log \frac{em}{s} \leq C_2 n R_q^{\frac{2}{q}}.
   \end{equation}
   There exists some constants $C_3,C_4>0$ such that for any $r = \Omega(\sqrt{\frac{s\log\frac{em}{s}}{n}}+\sqrt{sR_q}(\frac{\log d}{n})^{\frac{1}{2}-\frac{q}{4}})$, we have
  \begin{align*}
    \sup_{\theta \in \widetilde{\mathcal{S}}(\Theta^{m,d}_q(2s, 2R_q), r)} |\langle \bar Z, \theta \rangle| \leq \left(\sqrt{\frac{s \log \frac{em}{s}}{n}}+\sqrt{s R_q}(\frac{\log d}{n})^{\frac{1}{2}-\frac{q}{4}}\right)r.
  \end{align*} 
  with probability at least $1-C_3 \exp\{-C_4 n (\frac{s\log \frac{em}{s}}{n}+sR_q(\frac{\log d}{n})^{1-\frac{q}{2}})\}$.

\end{lemma}

\begin{proof}
    Similar to Lemma 7 in \cite{raskutti2011minimax},  we want to construct a constant $\delta$ that satisfies
  \begin{equation}\label{ap8}
    \sqrt{n} \delta \geq C_{1} r \\
  \end{equation}
  and
  \begin{align}\label{ap9}
    \begin{split}
    C_{2}\sqrt{n} \delta &\geq  \int_{\frac{\delta}{16}}^{r} \sqrt{\log N\left(t ; \widetilde{\mathcal{S}}(\Theta^{m,d}_q(2s, 2R_q), r),\|\cdot\|_F\right)} d t \coloneqq J(r, \delta).
  \end{split}
  \end{align}
  As long as $\|\bar Z\|_F^2 \leq 16$, Lemma 3.2 in \cite{geer2000empirical} guarantees that 
  \begin{align*}
    P\left ( \sup_{\theta \in \widetilde{\mathcal{S}}(\Theta^{m,d}_q(2s, 2R_q), r)} |\langle \bar Z, \theta\rangle| \geq \delta, \|\bar Z\|_2^2\leq 16\right )
    \leq C_3 \exp(-C_4 \frac{n\delta^2}{r^2}).
  \end{align*}
  Note that each entry of $\bar Z$ draws from $N(0, \frac{\sigma^2}{n})$. By the tail bounds of $\chi^2$ random variables in \cite{raskutti2011minimax}, we have $P(\|\bar Z\|_F^2 \geq 16) \leq C_5 \exp(-C_6n)$.
  Consequently, we have
  \begin{align*}
    P\left ( \sup_{\theta \in \widetilde{\mathcal{S}}(\Theta^{m,d}_q(2s, 2R_q), r)} |\langle \bar Z, \theta \rangle| \geq \delta\right ) \leq C_3 \exp(-C_4 \frac{n\delta^2}{r^2})+C_5 \exp(-C_6n).
  \end{align*}
  Next, we construct $\delta$ to satisfy conditions $\eqref{ap8}$ and $\eqref{ap9}$.
  Let $\delta = r(\sqrt{\frac{s\log \frac{em}{s}}{n}}+\omega)$, where $\omega>0$ is some constnt to be determined latter.
  Obviously, \eqref{ap8} holds immediately. Turning to \eqref{ap9}, 
  based on the condition \eqref{ap15}, we set $r = \Omega(\sqrt{\frac{s\log\frac{em}{s}}{n}}+\sqrt{sR_q}(\frac{\log d}{n})^{\frac{1}{2}-\frac{q}{4}})\wedge \sqrt{s}R_q^{\frac{1}{q}} $, and it is straightforward to verify that $(\frac{\delta}{16}, r)$ lies in the range of $\varepsilon$ in Lemma \ref{lem3}.
  Combined with the part (b) of Lemma \ref{lem3},
  we have 
  \begin{align*}
    J(r, \delta) &= \int_{\frac{\delta}{16}}^{r} \sqrt{\log N\left(t ; \widetilde{\mathcal{S}}(\Theta^{m,d}_q(2s, 2R_q), r)\right)}\ d t\\
    &\leq \int_{0}^{r} \sqrt{2s\log\frac{em}{s}+2s(\frac{s}{t^2}R_q^{\frac{2}{q}})^{\frac{q}{2-q}} \log d}\ d t\\
    &\leq \sqrt{2s \log \frac{em}{s}} r+\sqrt{2}(s R_q)^{\frac{1}{2-q}}\sqrt{\log d} r^{1-\frac{q}{2-q}}.
  \end{align*}
  With our choice of $\delta$, we have
  \begin{equation*}
    \frac{J(r, \delta)}{\sqrt{n}\delta} \leq \frac{\sqrt{2s \log \frac{em}{s}} r+\sqrt{2}(s R_q)^{\frac{1}{2-q}}\sqrt{\log d} r^{1-\frac{q}{2-q}}}{r\sqrt{s\log\frac{em}{s}}+r\sqrt{n}\omega}.
  \end{equation*}
  Setting $\omega = \sqrt{2}(s R_q)^{\frac{1}{2-q}}\sqrt{\frac{\log d}{n}} r^{1-\frac{q}{2-q}}$, we obtain $\frac{J(r, \delta)}{\sqrt{n}\delta} \leq \sqrt{2}$, which implies that \eqref{ap9} holds.
  
  Consequently, we obtain that $\delta = \left(\sqrt{\frac{s\log \frac{em}{s}}{n}}+\sqrt{sR_q}(\frac{\log d}{n})^{\frac{1}{2}-\frac{q}{4}}\right)r$. Overall, we conclude that
  \begin{align*}
    &P\left(\sup_{\theta \in \widetilde{\mathcal{S}}(\Theta^{m,d}_q(2s, 2R_q), r)} |\langle \bar Z, \theta \rangle| \geq \left(\sqrt{\frac{s \log \frac{em}{s}}{n}}+\sqrt{s R_q}(\frac{\log d}{n})^{\frac{1}{2}-\frac{q}{4}}\right)r \right )\\
\leq& C_3 \exp\{-C_4 n (\frac{s\log \frac{em}{s}}{n}+sR_q(\frac{\log d}{n})^{1-\frac{q}{2}})\},
  \end{align*} 
which completes the proof of Lemma \ref{chaining}.
  \QEDB
\end{proof}

\subsection{Proof of Theorem 2}\label{proof2}
We consider the constrained MLE estimator in \eqref{e7}. Note that for $q \in [0, 1]$,
  \begin{equation*}
    \|\bar Y-\hat \theta_q\|_F^2 \leq \|\bar Y - \theta^*\|_F^2.
  \end{equation*}
  By some simple algebras, we have
  \begin{equation}\label{ap7}
    \|\hat\theta_q - \theta^*\|_F^2 \leq 2 |\langle \bar Z, \hat\theta_q-\theta^*\rangle|.
  \end{equation}

 \subsubsection{Proof of case(a)}
 \begin{proof}\label{proof4}
  For case (a),
  since $\hat \theta_0, \theta^* \in \Theta^{m,d}_0(s, s_0)$, we have $\theta_0 - \theta^* \in \Theta^{m,d}_0(2s, 2s_0)$. 
  From Lemma \ref{covering}, for any $r > 0$, we have
  \begin{equation*}
      \sup_{\theta \in \widetilde{\mathcal{S}}(\Theta^{m,d}_0(2s, 2s_0), r)} |\langle \bar Z, \theta \rangle| \leq C_u \sigma r \sqrt{\frac{1}{n}(s \log \frac{em}{s}+ss_0\log \frac{ed}{s_0})}
    \end{equation*}
    with probability greater than $1-C_1\exp \{-C_2 (s \log \frac{em}{s}+ss_0\log \frac{ed}{s_0})\}$.
      
  Consider the event $\mathcal{Z}$ that there exists some $\theta$ satisfying $\theta \in \Theta^{m,d}_0(2s, 2s_0)$ such that
  \begin{equation}\label{ap14}
    |\langle \bar Z, \theta \rangle| \geq C_u \sigma \|\theta\|_F \sqrt{\frac{1}{n}(s \log \frac{em}{s}+ss_0\log \frac{ed}{s_0})}.
  \end{equation}
  Following from Lemma \ref{peeling}, 
  we have 
  \begin{equation*}
    P(\mathcal{Z}) \leq \frac{2\exp(-C_3 (s \log \frac{em}{s}+ss_0\log \frac{ed}{s_0}))}{1- \exp(-C_3 (s \log \frac{em}{s}+ss_0\log \frac{ed}{s_0}))}.
  \end{equation*}
  This claim follows from Lemma \ref{peeling} by choosing the function $f(v ; X)=\langle\bar Z, \theta\rangle$, the set $A=\Theta^{m,d}_0(2s, 2s_0)$, the sequence $a_{n}=n$, and the functions $\rho(v)=\|v\|_{2}$, and $g(r)=C_u\sigma r \sqrt{\frac{1}{n}(s \log \frac{em}{s}+ss_0\log \frac{ed}{s_0})}$. For any $r \geq \sigma \sqrt{\frac{1}{n}(s \log \frac{em}{s}+ss_0\log \frac{ed}{s_0})}$, we are guaranteed that $g(r) \geq  \frac{\sigma^{2}}{n}(s \log \frac{em}{s}+ss_0\log \frac{ed}{s_0})$ and $\mu = \frac{\sigma^{2}}{n}(s \log \frac{em}{s}+ss_0\log \frac{ed}{s_0})$, so that the lemma can be applied.

  Consequently, by \eqref{ap7} and \eqref{ap14}, some algebra shows that
  \begin{equation*}
    \|\hat\theta_0-\theta^*\|_F^2 \leq C_u\frac{\sigma^2}{n}( s \log \frac{em}{s}+ss_0\log \frac{ed}{s_0})
  \end{equation*}
      with probability greater than $1-C_1\exp \{-C_2 (s \log \frac{em}{s}+ss_0\log \frac{ed}{s_0})\}$. This completes the proof of \eqref{eq:u1}.
\end{proof}
\subsubsection{Proof of case(b)}

\begin{proof}
  For case (b),
  since $\hat \theta_q, \theta^* \in \Theta^{m,d}_q(s, R_q)$, we have $\hat \theta_q - \theta^* \in \Theta^{m,d}_q(2s, 2R_q)$.  
 
We define the event $\mathcal{Z}$ as there exists some $\theta$ such that $\theta \in \Theta^{m,d}_q(s, R_q)$ so that by Lemma \ref{chaining}, we have
  \begin{equation*}
    |\langle \bar Z, \theta\rangle| \geq C_u \|\theta\|_F\left(\sqrt{\frac{s \log \frac{em}{s}}{n}}+\sqrt{s R_q}(\frac{\log d}{n})^{\frac{1}{2}-\frac{q}{4}}\right).
  \end{equation*}
with probability at least $1-C_3 \exp\{-C_4 n (\frac{s\log \frac{em}{s}}{n}+sR_q(\frac{\log d}{n})^{1-\frac{q}{2}})\}$.  From Lemma \ref{peeling}, we have
  \begin{equation}\label{ap11}
    P(\mathcal{Z}) \leq \frac{2 \exp(-C_3n (\frac{s\log \frac{em}{s}}{n}+sR_q(\frac{\log d}{n})^{1-\frac{q}{2}}))}{1 - \exp(-C_3n (\frac{s\log \frac{em}{s}}{n}+sR_q(\frac{\log d}{n})^{1-\frac{q}{2}}))}.
  \end{equation}
  This claim follows from Lemma \ref{peeling} by choosing the function $f(v ; X)=\langle \bar Z, \theta\rangle$, the set $A=\Theta^{m,d}_q(2s, 2R_q)$, the sequence $a_{n}=n$, and the functions $\rho(v)=\|v\|_{F}$, and $g(r)=r\left(\sqrt{\frac{s \log \frac{em}{s}}{n}}+\sqrt{s R_q}(\frac{\log d}{n})^{\frac{1}{2}-\frac{q}{4}}\right)$. 

  Combining \eqref{ap7} and these results, we have
  \begin{align*}
    \|\hat \theta - \theta^*\|_F^2 \leq C_u \left( \sqrt{\frac{s \log \frac{em}{s}}{n}}+\sqrt{s R_q}(\frac{\log d}{n})^{\frac{1}{2}-\frac{q}{4}}\right),
  \end{align*}
  with probability at least $1-C_3\exp \{-C_4n (\frac{s \log \frac{em}{s}}{n}+ s R_q(\frac{\sigma^2v }{4n}\log d)^{1-\frac{q}{2}})\}$, which completes the proof of \eqref{eq:u2}.
  
  \QEDB
\end{proof}
\subsection{Proof of Theorem \ref{th4}}
\subsubsection{Lower bound for $\ell_{u}(\ell_{q})$}

\begin{proof}
    On the basis of double $\ell_0$ and $\ell_0$-$\ell_q$ pattern, we need two parameters to be determined, i.e., the group sparsity $1\le s \le m$ and the within-group sparsity $1\le s_0 \le d$. The use of $s$ and $s_0$ is the same as the previous proofs. We first construct the subspace of $\Theta^{m,d}_0 (s, s_0)$ similar to Lemma \ref{lem2}. Then, for each element of this subspace, we set the absolute value of the non-zero entries equal to $\delta$. 
    Consequently,
$$
s(s_0 \delta^{q})^{\frac{u}{q}} = R\implies \delta = R^{\frac{1}{u}}s^{-\frac{1}{u}}s_0^{-\frac{1}{q}}.
$$
By the generalized Fano's inequality, we have
\begin{equation*}
	\inf_{\hat \theta} \sup_{\theta^* \in \Theta^{m,d}_{u,q}(R)}P\left(\|\hat{\theta}-\theta^*\|_F^2 \ge \frac{1}{4}R^{\frac{2}{u}}s^{1-\frac{2}{u}}s_0^{1-\frac{2}{q}}\right) \geq 1 - \frac{\frac{n}{\sigma^2}R^{\frac{2}{u}}s^{1-\frac{2}{u}}s_0^{1-\frac{2}{q}}+\log 2}{\frac{1}{4}(ss_0\log\frac{d}{s_0}+s\log\frac{m}{s})},
\end{equation*}
Ignoring the constant factors, the above problem can be solved by linear programming: Let $y = \log s,\ x = \log s_0$, our target is maximizing 
$$
w = (1-\frac{2}{u})y+ (1-\frac{2}{q})x,
$$
and the  feasible region is given by:
$$
\left\{\begin{array}{ll}
	0 \le x \le \log d\\
	0 \le y \le \log m \\
 y \ge \min\left\{-\frac{u}{q}x  +\log R + \frac{u}{2}(\log n - \log(\log d)), 	- (\frac{u}{q}-\frac{u}{2})x+ \log R + \frac{u}{2}(\log n - \log(\log d))\right\}.
\end{array}\right.
$$

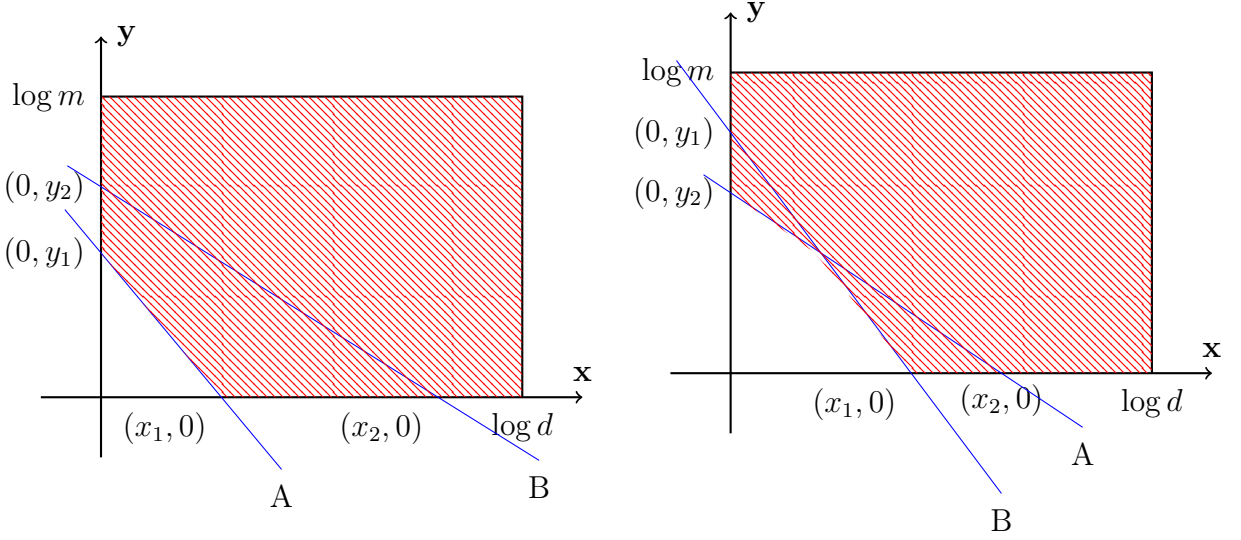
\begin{figure}
  \centering
\begin{minipage}{0.45\textwidth}
\begin{tikzpicture}[scale = 0.8]
\coordinate (X) at (8,0);\coordinate (Y) at (0,6);
\draw[-to,thick] (-1,0) to (X);
\node[above=2] at (X) {$\mathbf{x}$};
\draw[-to,thick] (0,-1) to (Y);
\node[right=2] at (Y) {$\mathbf{y}$};
\coordinate (UR) at (7,5);\coordinate (LR) at (7,0);\coordinate (UL) at (0,5);
\draw[thick] (LR) -- (UR) -- (UL);
\node[left=2,] at (UL) {$\log m$};
\node[below=2] at (LR) {$\log d$};
\coordinate (x1) at (2,0);\node[below left=1.5] at (x1) {$(x_1,0)$};
\coordinate (x2) at (5.6,0);\node[below left=1.5] at (x2) {$(x_2,0)$};
\coordinate (y2) at (0,2.4);\node[left=2] at (y2) {$(0,y_1)$};
\coordinate (y1) at (0,3.5);\node[left=2] at (y1) {$(0,y_2)$};
\draw[color=blue] ($1.3*(y2)-0.3*(x1)$) -- ($1.5*(x1)-0.5*(y2)$);\node[below=2] at ($1.5*(x1)-0.5*(y2)$) {A};
\draw[color=blue] ($1.1*(y1)-0.1*(x2)$) -- ($1.3*(x2)-0.3*(y1)$);\node[below=2] at ($1.3*(x2)-0.3*(y1)$) {B};
\fill[pattern=north west lines, pattern color=red] (y2) -- (x1) -- (LR) -- (UR) -- (UL) -- (y2);
\end{tikzpicture}
\end{minipage}
\hspace{0.05\textwidth}
\begin{minipage}{0.45\textwidth}
\begin{tikzpicture}[scale = 0.8]
\coordinate (X) at (8,0);\coordinate (Y) at (0,6);
\draw[-to,thick] (-1,0) to (X);
\node[above=2] at (X) {$\mathbf{x}$};
\draw[-to,thick] (0,-1) to (Y);
\node[right=2] at (Y) {$\mathbf{y}$};
\coordinate (UR) at (7,5);\coordinate (LR) at (7,0);\coordinate (UL) at (0,5);
\draw[thick] (LR) -- (UR) -- (UL);
\node[left=2,] at (UL) {$\log m$};
\node[below=2] at (LR) {$\log d$};
\coordinate (x2) at (4.5,0);\node[below] at (x2) {$(x_2,0)$};
\coordinate (x1) at (3,0);\node[below left=1.5] at (x1) {$(x_1,0)$};
\coordinate (y2) at (0,3);\node[left=2] at (y2) {$(0,y_2)$};
\coordinate (y1) at (0,4);\node[left=2] at (y1) {$(0,y_1)$};
\draw[color=blue] ($1.3*(y1)-0.3*(x1)$) -- ($1.5*(x1)-0.5*(y1)$);\node[below=2] at ($1.5*(x1)-0.5*(y1)$) {B};
\draw[color=blue] ($1.1*(y2)-0.1*(x2)$) -- ($1.3*(x2)-0.3*(y2)$);\node[below=2] at ($1.3*(x2)-0.3*(y2)$) {A};
\fill[pattern=north west lines, pattern color=red] (y2) -- (1.4, 2) -- (x1) -- (LR) -- (UR) -- (UL) -- (y2);
\end{tikzpicture}
  \end{minipage}  
  \caption{Plot of the feasible regions. The red region represents the feasible region, and the blue lines represent the Line A and Line B, respectively. 
  Line A intersects the x-axis and y-axis at $(x_1, 0)$ and $(0, y_1)$, respectively, and Line B intersects the x-axis and y-axis at $(x_2, 0)$ and $(0, y_2)$, respectively.}
\label{fig:f7}
\end{figure}
To simplify the notations, we denote $x_1$ as $\log R + \frac{u}{2}\log(\frac{n}{\log m})$, $x_2$ as
$\log R + \frac{u}{2}\log(\frac{n}{\log d})$, $y_1$ as $\frac{q}{u}\log R + \frac{q}{2}\log(\frac{n}{\log m})$ and $y_2$ as $\frac{q}{u}\log R + \frac{q}{2}\log(\frac{n}{\log d})$.
To provide a clearer explanation of the problem, we present the possible feasible regions in Figure \ref{fig:f7}.
It is worth noting that the slope of Line A is calculated as $|\frac{\frac{2}{q}}{\frac{2}{u}}|$, while the slope of Line B is determined by $|\frac{\frac{2}{q}-1}{\frac{2}{u}}|$.
Consequently, the slope of Line A is greater than the slope of Line B.
On the other hand, the slope of the objective function is given by $|\frac{\frac{2}{q}-1}{\frac{2}{u}-1}|$, which is larger than the slopes of both Line A and Line B when $u > q$. However, when $u \leq q$, the slope of the objective function is smaller than that of Line A but greater than that of Line B.
Therefore, maximum point for $w$ is chosen from $(0,\log R + \frac{u}{2}\log(\frac{n}{\log d})),(0,\log R + \frac{u}{2}\log(\frac{n}{\log m})),(\frac{q}{u}\log R + \frac{q}{2}\log(\frac{n}{\log d}),0)$.

So that the minimax lower bound is given by:
$$
\left\{\begin{array}{ll}
	R\left(\frac{n}{\log d}\right)^{\frac{u-2}{2}}+ R\left(\frac{n}{\log m}\right)^{\frac{u-2}{2}} & u > q\\
	R^{\frac{q}{u}}\left(\frac{n}{\log d}\right)^{\frac{q-2}{2}}+R\left(\frac{n}{\log m}\right)^{\frac{u-2}{2}} & u \le q,~ m \ge d\\
	R^{\frac{q}{u}}\left(\frac{n}{\log d}\right)^{\frac{q-2}{2}}  & u \le q,~ m \le d
\end{array}\right.
$$
\QEDB
\end{proof}
\subsubsection{Upper bound for $\ell_{u}(\ell_{q})$}

In order to obtain the covering number of $\mathbb{B}(R)$ which is equipped with $\ell_{u}(\ell_{q})$-norm, we need a generalization of Sch{\"u}tt’s theorem for vector-valued sequence spaces \cite{edmunds2014schutt}.

\begin{lemma}[Theorems 3.1 and 3.2 in \cite{edmunds2014schutt}]
	Let $X, Y$ be r-normed quasi-Banach spaces, and $0 < q < r \le \infty$, unit ball $\mathbb{B}_{\ell_q^m(X)}$ denotes
	$$
	\mathbb{B}_{\ell_q^m(X)} = v_1\mathbb{B}_X \times v_2\mathbb{B}_X \times \cdots v_m \mathbb{B}_X,
	$$ 
where $\mathbb{B}_X$ is unit ball with $X$-norm, and $m$-dimensional vector $v$ belongs to unit ball $\mathbb{B}_{q}$. For $k, k_0 \in \mathbb{N}$ such that $k_0 \le k$, let
$$
D(k_0,k) = \max_{l\in \mathbb{N}, k_0 \le l \le k}\left(\frac{l}{k}\right)^{\frac{1}{q}-\frac{1}{r}}e_{l}(id: X\to Y),
$$
and
$$
A(k,m) = \max\left\{\|id: X\to Y\|\left(\frac{\log(em/k) }{k}\right)^{\frac{1}{q}-\frac{1}{r}} ,D(1,k)\right\},
$$
where $\|id: X\to Y\|$ denotes the operator norm between $X$ and $Y$ and $e_{l}(id: X\to Y)$ denotes entropy number. For $k \ge \log_2(m)$, we have the following,
\begin{itemize}
	\item If $k \le m$, then
	$$
	e_k\left(id: \ell_q^m(X) \to \ell_r^m(Y)\right)\simeq A(k,m).
	$$
	\item If $k \ge m$, then there are absolute constants $C_1, C_2$ such that
	$$
	D(C_1k/m, k) \le e_k\left(id: \ell_q^m(X)\to \ell_r^m(Y)\right) \le D(C_2k/m, k).
	$$
\end{itemize}
\end{lemma}

Let $q = u$, $r = 2$, $X = \ell_{q}^d$, or $Y = \ell_{2}^d$ for our regression problem, so that $\|id: X\to Y\| = 1$ and $e_{l}(id: X\to Y)$ is given by \ref{wzero}, \ref{high} and \ref{low}. We also inherit the spirit of \cite{raskutti2011minimax} by only considering $ \log m < k < m$ and $\log d \le k \le d$, just like we did for $\ell_0(\ell_q)$ case, we will present it is reasonable for condition \eqref{lqqcondition}. Therefore it can be divided into three cases:
\begin{enumerate}
	\item[\bf{Case(a)}:] for  $u\le q, ~ m \le d$, we assume $\log(md)\le k \le d$, then
	$$
	e_k \simeq 	\left(\frac{\log(ed)}{k}\right)^{\frac{1}{q}-\frac{1}{2}}.
	$$
\item[\bf{Case(b)}:] for  $u\le q, ~ m \ge d$, we assume $\log(md)\le k \le m$, then
	$$
	e_k \simeq \left\{\begin{array}{ll}
	\max\left\{\left(\frac{\log(em)}{k}\right)^{\frac{1}{u}-\frac{1}{2}},\left(\frac{\log (ed)}{k}\right)^{\frac{1}{q}-\frac{1}{2}}\right\}, & \log(md) \le k \le d\\
	\max\left\{\left(\frac{\log(em)}{k}\right)^{\frac{1}{u}-\frac{1}{2}},\left(\frac{d}{k}\right)^{\frac{1}{u}-\frac{1}{2}} d^{\frac{1}{2}-\frac{1}{q}}\right\}, & d\le k \le m.
\end{array}\right.
	$$
Notably, when $k \ge d$, we observe that $(\log d)^{\frac{1}{q}-\frac{1}{2}}\cdot (\frac{k}{d})^{\frac{1}{u}-\frac{1}{q}}>1$, then we have
$$
e_k \le \max\left\{\left(\frac{\log(em)}{k}\right)^{\frac{1}{u}-\frac{1}{2}},\left(\frac{\log (ed)}{k}\right)^{\frac{1}{q}-\frac{1}{2}}\right\} ~ \text{for}~ \log(md) \le k \le m.
$$
\item[\bf{Case(c)}:] for $u > q$,
	$$
e_k \simeq \left\{\begin{array}{ll}
\left(\frac{\log (emd)}{k}\right)^{\frac{1}{u}-\frac{1}{2}} & \log(md) \le k \le m\log d\\
m^{\frac{1}{q}-\frac{1}{u}}\left(\frac{\log (emd)}{k}\right)^{\frac{1}{q}-\frac{1}{2}} & m\log d\le k \le md.
\end{array}\right.
$$
When $u > q$ and $k \ge m\log d$. We assume $k \ge m\log(emd)$, then we have
$$
\left(\frac{\log (emd)}{k}\right)^{\frac{1}{u}-\frac{1}{2}} \ge m^{\frac{1}{q}-\frac{1}{u}}\left(\frac{\log (emd)}{k}\right)^{\frac{1}{q}-\frac{1}{2}}.
$$
Therefore, we have
$$
e_k \le \left(\frac{\log (emd)}{k}\right)^{\frac{1}{u}-\frac{1}{2}} \text{, for}\  \log(md) \le k \le \max\{d,m\}.
$$
\end{enumerate}

Without loss of generality, we assume that for any $\epsilon>0$, $\log(md) \le d^\epsilon$ and $\log(md) \le m^{\epsilon}$. In conclusion, for given $\log(md) \le k \le \max\{d,m\} $ we have
\begin{equation}\label{lqcovering}
	e_k \le \left\{\begin{array}{ll}
		\left(\frac{\log (ed)}{k}\right)^{\frac{1}{q}-\frac{1}{2}} & u\le q, ~ m \le d\\
		\max\left\{\left(\frac{\log(ed)}{k}\right)^{\frac{1}{q}-\frac{1}{2}},\left(\frac{\log (em)}{k}\right)^{\frac{1}{u}-\frac{1}{2}}\right\} & u\le q, ~ m \ge d\\
		\left(\frac{\log (emd)}{k}\right)^{\frac{1}{u}-\frac{1}{2}} & q\le u.
	\end{array}\right.
\end{equation}
Notably, this rate is tight for $\log(md) \le k \le \min\{d,m\}$.
Now we follow the same procedure as $\ell_{0}(\ell_q)$ case. Define 
$$
\widetilde{\mathcal{S}}(\Theta^{m,d}_{u,q}(R), r) = \left\{\theta \in \mathbb{R}^{d\times m}:\|\theta\|_F \le r\right\}\cap \Theta^{m,d}_{u,q}(R).
$$

\begin{proof}
	
 For case $u \le q , m \ge d$, we add radius $R^{\frac{1}{u}}$ into \eqref{lqcovering} and obtain
	 $$
	 \epsilon \le R^{\frac{1}{u}}\left(\frac{\log(ed)}{k}\right)^{\frac{1}{q}-\frac{1}{2}}+R^{\frac{1}{u}}\left(\frac{\log (em)}{k}\right)^{\frac{1}{u}-\frac{1}{2}}.
	 $$
By solving the above inequality, we have
\begin{equation}\label{lqqcovering}
	\log N\left(\epsilon ; \widetilde{\mathcal{S}}(\Theta^{m,d}_{u,q}(R), r)\right)\le  \log (ed)(\epsilon^{-1}R^{\frac{1}{u}})^{\frac{2q}{2-q}} + 
	\log (em)(\epsilon^{-1}R^{\frac{1}{u}})^{\frac{2u}{2-u}}.
\end{equation}
	 
	Next,  we follow the Lemma 3.2 of \cite{geer2000empirical} by constructing  constants $(\delta,r)$ that satisfies
	 \begin{equation}\label{aq8}
	 	\sqrt{n} \delta \geq C_{1} r \\
	 \end{equation}
	 and
	 \begin{align}\label{aq9}
	 	\begin{split}
	 		C_{2}\sqrt{n} \delta \geq  J(r, \delta).
	 	\end{split}
	 \end{align}
Note that
\begin{align*}
  J(r, \delta) &= \int_{\frac{\delta}{16}}^{r} \sqrt{\log N\left(t ; \widetilde{\mathcal{S}}(\Theta^{m,d}_{u,q}(R), r)\right)}\ d t\\
  &\leq \int_{0}^{r} \sqrt{ \log (ed)(t^{-1}R^{\frac{1}{u}})^{\frac{2q}{2-q}} +\log (em)(t^{-1}R^{\frac{1}{u}})^{\frac{2u}{2-u}}}\ d t\\
  &\leq \sqrt{\log(ed)}R^{\frac{q}{u(2-q)}}r^{1-\frac{q}{2-q}} + \sqrt{\log(em)}R^{\frac{1}{2-u}}r^{1-\frac{u}{2-u}} .
\end{align*}
First, we set 
\begin{equation}\label{lqqr}
	r = \Omega\left(	R^{\frac{q}{2u}}\left(\frac{n}{\log d}\right)^{\frac{q-2}{4}}+R^{\frac{1}{2}}\left(\frac{n}{\log m}\right)^{\frac{u-2}{4}}\right)\wedge R^{\frac{1}{u}}
\end{equation}
and 
\begin{equation}\label{lqqdelta}
	\delta = Cr\left(	R^{\frac{q}{2u}}\left(\frac{n}{\log d}\right)^{\frac{q-2}{4}}+R^{\frac{1}{2}}\left(\frac{n}{\log m}\right)^{\frac{u-2}{4}}\right).
\end{equation}
Next, we  clarify that $(\delta,r)$ satisfies conditions $\eqref{aq8}$ and $\eqref{aq9}$. For $\eqref{aq8}$, we have
$$
\left(\frac{\sqrt{n}\delta}{r}\right) \ge C\left(R^{\frac{q}{u}}n^{\frac{q}{2}}\left(\frac{1}{\log d}\right)^{\frac{q-2}{2}}+Rn^{\frac{u}{2}}\left(\frac{1}{\log m}\right)^{\frac{u-2}{2}}\right) \ge C_1.
$$
The second inequality is followed by condition \eqref{lqqcondition}.
For $\eqref{aq9}$, we have
$$
\begin{aligned}
	\frac{J(r, \delta)}{\sqrt{n}\delta} = \frac{\sqrt{\log(ed)}R^{\frac{q}{u(2-q)}}r^{-\frac{q}{2-q}} + \sqrt{2\log(m)}R^{\frac{1}{2-u}}r^{-\frac{u}{2-u}} }{C\sqrt{n}\left(	R^{\frac{q}{2u}}\left(\frac{n}{\log d}\right)^{\frac{q-2}{4}}+R^{\frac{1}{2}}\left(\frac{n}{\log m}\right)^{\frac{u-2}{4}}\right)}.
\end{aligned}
$$
For $r = \Omega\left(	R^{\frac{q}{2u}}\left(\frac{n}{\log d}\right)^{\frac{q-2}{4}}+R^{\frac{1}{2}}\left(\frac{n}{\log m}\right)^{\frac{u-2}{4}}\right)$, we have

$$
\begin{aligned}
	\frac{J(r, \delta)}{\sqrt{n}\delta} &\le \frac{\sqrt{2\log(d)}R^{\frac{q}{u(2-q)}}[R^{\frac{q}{2u}}\left(\frac{n}{\log d}\right)^{\frac{q-2}{4}}]^{-\frac{q}{2-q}} + \sqrt{2\log(m)}R^{\frac{1}{2-u}}[R^{\frac{1}{2}}\left(\frac{n}{\log m}\right)^{\frac{u-2}{4}}]^{-\frac{u}{2-u}}}{C\left(	R^{\frac{q}{2u}}n^{\frac{q}{4}}\left(\frac{1}{\log d}\right)^{\frac{q-2}{4}}+R^{\frac{1}{2}}n^{\frac{u}{4}}\left(\frac{1}{\log m}\right)^{\frac{u-2}{4}}\right)}\\
	&=\frac{\sqrt{2}\sqrt{\log d}R^{\frac{q}{2u}}\left(\frac{n}{\log d}\right)^{\frac{q}{4}}+\sqrt{2}\sqrt{\log m}R^{\frac{1}{2}}\left(\frac{n}{\log m}\right)^{\frac{u}{4}}}{C\left(	R^{\frac{q}{2u}}n^{\frac{q}{4}}\left(\frac{1}{\log d}\right)^{\frac{q-2}{4}}+R^{\frac{1}{2}}n^{\frac{u}{4}}\left(\frac{1}{\log m}\right)^{\frac{u-2}{4}}\right)}\\
	&=\frac{\sqrt{2}}{C}.
\end{aligned}
$$
We end the procedure by clarifying $(\frac{\delta}{16}, r)$ make sense.
Following $r \ge R^{\frac{q}{2u}}\left(\frac{n}{\log d}\right)^{\frac{q-2}{4}}+R^{\frac{1}{2}}\left(\frac{n}{\log m}\right)^{\frac{u-2}{4}}$
and \eqref{lqqcovering}, we have:
$$
\begin{aligned}
	\log N\left(\delta ; \widetilde{\mathcal{S}}(\Theta^{m,d}_{u,q}(R), r)\right)&\le  \log (ed)\left(R^{-\frac{q}{2u}}\left(\frac{n}{\log d}\right)^{-\frac{q-2}{4}}R^{\frac{1}{u}}\right)^{\frac{2q}{2-q}} \\
	&+ 
\log (em)\left(R^{-\frac{1}{2}}\left(\frac{n}{\log m}\right)^{-\frac{u-2}{4}}R^{\frac{1}{u}}\right)^{\frac{2u}{2-u}}\\
&\le \log (ed)R^{\frac{q}{u}}(\frac{n}{\log})^{\frac{q}{2}}+\log (em)R(\frac{n}{\log})^{\frac{u}{2}}\\
&\le d+ m .
\end{aligned}
$$
The last inequality is followed by \eqref{lqqcondition}. 
Combined with $r < R^{\frac{1}{u}}$, we derive $\log N\left(\delta ; \widetilde{\mathcal{S}}(\Theta^{m,d}_{u,q}(R), r)\right) \ge \log(md)$ for $\delta < r$, and we have
$$
\left\{\log N\left(\delta\right): \delta \in  (\frac{\delta}{16}, r)\right\} \cap (\log (md),d+m] \ne \emptyset.
$$
Therefore, with probability at least $1-C_5 \exp\left\{-C_6 n \left(	R^{\frac{q}{u}}\left(\frac{n}{\log d}\right)^{\frac{q-2}{2}}+R\left(\frac{n}{\log m}\right)^{\frac{u-2}{2}}\right)\right\}$, we have
\begin{equation}\label{entropylqq1}
		 \begin{aligned}
		\sup_{\theta \in \widetilde{\mathcal{S}}(\Theta^{m,d}_{u,q}(R), r)} |\langle \bar Z, \theta \rangle| \leq \left(	R^{\frac{q}{2u}}\left(\frac{n}{\log d}\right)^{\frac{q-2}{4}}+R^{\frac{1}{2}}\left(\frac{n}{\log m}\right)^{\frac{u-2}{4}}\right)r.
	\end{aligned} 
\end{equation}
The other two cases that $u > q$ and $u \le q, m\le d$ are followed by the same procedure so we just give a concise sketch.

For case $u \le q , m \le d$,
 we set 
 $$
 r = \Omega\left(R^{\frac{q}{2u}}\left(\frac{n}{\log d}\right)^{\frac{q-2}{4}}\right)\wedge R^{\frac{1}{u}}
 $$
 and 
$$
\delta = C_1 r \cdot R^{\frac{q}{2u}}\left(\frac{n}{\log d}\right)^{\frac{q-2}{4}}.
$$
Therefore, with probability at least $1-C_7 \exp\left\{-C_8 n	R^{\frac{q}{u}}\left(\frac{n}{\log d}\right)^{\frac{q-2}{2}}\right\}$,
\begin{equation}\label{entropylqq2}
	\begin{aligned}
		\sup_{\theta \in \widetilde{\mathcal{S}}(\Theta^{m,d}_{u,q}(R), r)} |\langle \bar Z, \theta \rangle| \leq 	R^{\frac{q}{2u}}\left(\frac{n}{\log d}\right)^{\frac{q-2}{4}}r.
	\end{aligned} 
\end{equation}
For case $u \le q , m \le d$, 
we set 
$$
r = \Omega\left(R^{\frac{1}{2}}\left(\frac{n}{\log md}\right)^{\frac{u-2}{4}}\right)\wedge R^{\frac{1}{u}}
$$
and 
$$
\delta = C_1 r \cdot R^{\frac{1}{2}}\left(\frac{n}{\log md}\right)^{\frac{u-2}{4}}
$$
Therefore, with probability at least $1-C_2 \exp\left\{-C_{3} n 	R\left(\frac{n}{\log md}\right)^{\frac{u-2}{2}}\right\}$,
\begin{equation}\label{entropylqq3}
	\begin{aligned}
		\sup_{\theta \in \widetilde{\mathcal{S}}(\Theta^{m,d}_{u,q}(R), r)} |\langle \bar Z, \theta \rangle| \leq 	R^{\frac{1}{2}}\left(\frac{n}{\log d}\right)^{\frac{q_-2}{4}}r.
	\end{aligned} 
\end{equation}
We end the proof by plugging the concentration inequalities \eqref{entropylqq1}, \eqref{entropylqq2} and \eqref{entropylqq3} into the Lemma \ref{peeling}.
\QEDB
\end{proof}

\subsection{Proof of Corollary 1}
 Minimax rate for double sparse linear regression is a direct consequence of the Gaussian location model with the assumption on the design matrix. 
 We prove the lower bound and upper bound in the following part, respectively.
 \begin{itemize}
  \item[(a)] Lower bound: by Assumption \ref{ass1}(a), we derive the upper bound for the mutual inforamtion:
 $$
 I(Y;\psi)\le \frac{n}{2\sigma^2}\parallel X(\beta^i-\beta^j) \parallel^2_2 \le \frac{n}{\sigma^2}ss_0\tau_u^2 \delta'^2.
 $$ 
 The remaining proof is the same as the proof \ref{proof_lower} with setting $\delta = \tau_u \delta'$.
\item[(b)] Upper bound: We denote $\hat{\Delta} = \hat{\beta} - \beta^*$.
  Recall that the definition of the constrained maximum likelihood estimator:
\begin{align}
  \begin{split}
 \hat \beta \in \arg \min_{\beta \in \mathcal{B}^{m,d}_0(s, s_0)} \|Y-X\beta\|^2_2.
  \end{split}
\end{align}
We obtain $\hat{\Delta} \in \mathcal{B}^{m,d}_{0}(2 s, 2s_0)$. Consequently, by Assumption \ref{ass1}(b), we have
\begin{align*}
  \|\hat{\Delta}\|^2 &\le \frac{1}{\tau_{\ell}^2}\|X\hat{\beta} - X\beta \|^2_2\le \sup_{\hat{\Delta} \in \mathcal{B}^{m,d}_{0}(2 s, 2s_0)}\langle\xi,X^T\hat{\Delta}\rangle.
  \end{align*}
The remaining proof is the same as proof \ref{proof4} by using Lemma \ref{covering}.

 \end{itemize}

\subsection{Proof of Theorem 3}\label{proof3}

The proof of Lemma \ref{lem:dsrip} is a straightforward extension of the proof presented in Lemma 2 of \cite{ndaoud2020scaled}. For brevity, we omit the detailed procedure here.
\begin{lemma}[Lemma 2 of \cite{ndaoud2020scaled}]\label{lem5}
  For a given subset $S \in \mathcal{S}^{m, d}(s, s_0)$, we have
  \begin{align*}
    P\left(\frac{1}{n}\|X_{S}^\top \xi\|_2^2 \ge \sigma^2(ss_0+t)\right) \leq e^{-(t/8)\wedge(t^2/(64ss_0))},\ \forall t \ge 0.
    \end{align*}
\end{lemma}
Based on Lemma \ref{lem5}, Lemma \ref{lemma:iht1} is proved by union  bound and calculating that:
$$
|\mathcal{S}^{m,d}(s,s_0)| \le {m \choose s}\left\{{d \choose s_0}\right\}^s.
$$
\begin{proof}
  We proceed with the proof of Theorem 3 by induction. 
  The results are trivial for step $t = 0$. 
  Now we assume \eqref{eq:iht11} - \eqref{eq:iht4} hold for step $t$. Then, based on these assumptions, we prove that similar results hold for step $t+1$. 
  We first prove \eqref{eq:iht11} and \eqref{eq:iht3} by contradiction. 
  Assume that \eqref{eq:iht11} and \eqref{eq:iht3} are wrong for $t+1$, i.e., $(\mathop{\cup}_{j \in G^*})\cap S^{t+1} \cap (S^*)^c \notin \mathcal{S}^{m, d}(s, s_0)$ and $(\mathop{\cup}_{j \in G^*})^c\cap S^{t+1} \notin \mathcal{S}^{m, d}(s, s_0)$. 
  According to the procedure of our algorithm, we have
  $$
  \beta^{t+1} = \mathcal{T}_{\lambda_{t+1}, s ,s_0}(H^{t+1}).
  $$
  
  {\bf Proof of \eqref{eq:iht11} for step $t+1$:}
  We first prove \eqref{eq:iht11} by contradiction. 
  Assume that \eqref{eq:iht11} is wrong for $t+1$, which implies that when we consider the support $(\mathop{\cup}\limits_{j \in G^*} G_j) \cap(S^*)^c$, there exists a $(s, s_0)$-shape set $S_{1, t+1}$ satisfying that  
 $$
  ss_0 \lambda_{t+1}^2 \le \sum_{i \in S_{1, t+1}}\{\mathcal{T}_{\lambda_{t+1}}(H^{t+1})\}^2_i.
  $$
  This statement has been elaborated in {\bf{Case 1}}. 
  Since $\beta^*_i = 0$ for $i \in S_{1, t+1}$. Then, from the triangle inequality, we have
  $$
  \sqrt{ss_0}\lambda_{t+1} \le \sqrt{\sum_{i \in S_{1, t+1}}\langle \Phi_{i}^\top , \beta^*-\beta^{t}\rangle^2} + \sqrt{\sum_{i\in S_{1, t+1}}\Xi_{i}^2}.
  $$
  Note that $\beta^*$ is $(s,s_0)$-sparse and $\eqref{eq:iht11}, \eqref{eq:iht3}$ hold for step $t$ by induction.
  Consequently, $\beta^*-\beta^t$ is $(2s, 2s_0)$-sparse.
  Then, with probability at least $1-\exp\left\{-C\left(s\log(\frac{em}{s})+ss_0\log(\frac{ed}{s_0})\right)\right\}$, we have
  $$
  \begin{aligned}
    \sqrt{ss_0}\lambda_{t+1} \le& \frac{3+\sqrt{2}+2\sqrt{3}}{2}\delta \sqrt{ss_0}\lambda_t + \sqrt{\frac{10\sigma^2\left(s\log(\frac{em}{s})+ss_0\log(\frac{ed}{s_0})\right)}{n}}\\
    \le& \frac{3+\sqrt{2}+2\sqrt{3}}{2}\sqrt{\delta} \sqrt{ss_0}\lambda_{t+1}+\frac{1}{2}\sqrt{ss_0}\lambda_{\infty}\\
    \le& \left(\frac{3+\sqrt{2}+2\sqrt{3}}{2}\sqrt{\delta}+\frac{1}{2}\right)\sqrt{ss_0}\lambda_{t+1}\\\
    <& \sqrt{ss_0}\lambda_{t+1}
  \end{aligned}
  $$
  as long as $\delta < \frac{1}{(3+\sqrt{2}+2\sqrt{3})^2}$, which is absurd. 
  Here the first inequality employs the DSRIP condition and the assumption that result \eqref{eq:iht4} holds for step $t$, while the second inequality arises from the fact that $\lambda_{\infty} \leq \lambda_{t}$.  Therefore, \eqref{eq:iht11} has been proved for step $t+1$ by induction.

  {\bf Proof of \eqref{eq:iht3} for step $t+1$:} \eqref{eq:iht3} can be proved in a similar way as \eqref{eq:iht11}.
   Assume that \eqref{eq:iht3} is wrong for $t+1$, which implies that when we consider the support $(\mathop{\cup}\limits_{j \in G^*} G_j)^c$, there exists a $(s, s_0)$-shape set $S_{2, t+1}$ satisfying that  
 $$
  ss_0 \lambda_{t+1}^2 \le \sum_{i \in S_{2, t+1}}\{\mathcal{T}_{\lambda_{t+1}}(H^{t+1})\}^2_i.
  $$
  This statement has been elaborated in {\bf{Case 1}}-{\bf{Case 3}}. Therefore, we have
  $$
  \sqrt{ss_0}\lambda_{t+1} \le \sqrt{\sum_{i \in S_{2, t+1}}\langle \Phi_{i}^\top , \beta^*-\beta^{t}\rangle^2} + \sqrt{\sum_{i\in S_{2, t+1}}\Xi_{i}^2}.
  $$
  The remaining analysis is the same as the proof of \eqref{eq:iht11}.

  {\bf Proof of \eqref{eq:iht4} for step $t+1$:}
  Note that $\eqref{eq:iht11}$ and $\eqref{eq:iht3}$ hold for step $t+1$.
  We get that $\beta^{t+1} - \beta^*$ is $(2s, 2s_0)$-sparse.
  Observe that for any $i \in [p]$, 
  \begin{equation}\label{eq:iht12}
  (\beta^{t+1})_i - \beta^*_i = -H^{t+1}_i \mathbb{I}(i \notin S^{t+1}) + \langle \Phi_{i}^\top, \beta^*-\beta^{t}\rangle + \Xi_i.
\end{equation}
On one hand, for $i \in S^{t+1}\cap (S^*)^c$, $\beta^*_i = 0$.
Summing both sides of \eqref{eq:iht12} over set $S^{t+1}\cap (S^*)^c$, we have 
\begin{align}\label{eq:iht8}
  \begin{split}
  \| \beta^{t+1}_{(S^*)^c}\|_2 \le& \sqrt{\sum_{i\in S^{t+1}\cap (S^*)^c}\langle \Phi_{i}^\top, \beta^*-\beta^{t}\rangle^2}+\sqrt{\sum_{i\in S^{t+1}\cap (S^*)^c}\Xi_{i}^2}\\
  \le& \delta \| \beta^t-\beta^* \| +\sqrt{\frac{20\sigma^2\left(s\log(\frac{em}{s})+ss_0\log(\frac{ed}{s_0})\right)}{n}}.
\end{split}
\end{align}
On the other hand, summing both sides of \eqref{eq:iht12} over support set $S^*$, we have 
\begin{equation}\label{eq:iht9}
  \begin{aligned}
    \| (\beta^{t+1}-\beta^*)_{S^*} \|_2 \le& \sqrt{\sum_{i\in \cap S^*}\{H^{t+1}_i\mathbb{I}(i \notin S^{t+1})\}^2}  +\sqrt{\sum_{i\in S^*}\langle \Phi_i^\top, \beta^*-\beta^{t}\rangle^2}+ \sqrt{\sum_{i\in S^*}\Xi_i^2}\\
    \le& \sqrt{3ss_0}\lambda_{t+1} + \delta\| \beta^*-\beta^{t}\| +\sqrt{\frac{10\sigma^2\left(s\log(\frac{em}{s})+ss_0\log(\frac{ed}{s_0})\right)}{n}}.
  \end{aligned}
\end{equation}
Since operator $\mathcal{T}_{\lambda, s_0}(\cdot)$ has three thresholding procedure, i.e., element-wise thresholding and group thresholding with two respects, the term 
$\sum_{i\in S^*}(H^{t+1}_i)^2$ in \eqref{eq:iht9} is upper bounded by $3ss_0\lambda^2_{t+1}$
Combining \eqref{eq:iht8} and \eqref{eq:iht9}, we conclude that
$$
\begin{aligned}
  \|\beta^{t+1}-\beta^* \|_2 
  &\le \| \beta^{t+1}_{(S^*)^c}\|_2+\| (\beta^{t+1}-\beta^*)_{S^*} \|_2\\
  &\le \sqrt{3ss_0}\lambda_{t+1} +(1+\sqrt{2})\sqrt{\frac{10\sigma^2\left(s\log(\frac{em}{s})+ss_0\log(\frac{ed}{s_0})\right)}{n}}\\
  &\le \left(\sqrt{3}+(3+\sqrt{2}+2\sqrt{3})\sqrt{\delta}+\frac{1+\sqrt{2}}{2}\right)\sqrt{ss_0}\lambda_{t+1}\\
  &\le \frac{3+\sqrt{2}+2\sqrt{3}}{2}\sqrt{ss_0}\lambda_{t+1},
\end{aligned}
$$
where the third and the last inequalities follow from $\delta < \frac{1}{(3+\sqrt{2}+2\sqrt{3})^2}\wedge \kappa$.
We prove that $\eqref{eq:iht4}$ holds for step $t+1$.
Finally, we have proved that the results in Theorem \ref{th3} hold for $t+1$ under the induction hypothesis. This completes the proof of Theorem \ref{th3}.
  \QEDB
\end{proof}

\bibliographystyle{unsrtnat}
\bibliography{ref}

\end{document}